\documentclass[11pt, a4paper, reqno]{amsart}
\usepackage{amsmath,amsthm}
\usepackage{amssymb, enumitem, mathrsfs, verbatim, bm, esint}
\usepackage{amstext}
\usepackage[top=3cm,bottom=2.5cm,left=2.5cm,right=2.5cm]{geometry}

\usepackage{hyperref}
\usepackage{xcolor}
\usepackage{enumitem}

\newtheorem{theorem}{Theorem}[section]

\newtheorem{lemma}[theorem]{Lemma}
\newtheorem{definition}[theorem]{Definition}

\newcommand{\N}{\mathbb N}

\newcommand{\R}{{\mathbb R}}

\newcommand{\dual}[2]{\langle #1,#2\rangle}
\newcommand{\Fk}{\Phi_k}

\numberwithin{equation}{section}
\newcommand{\diver}{\operatorname{div}}

\DeclareMathOperator*{\supp}{supp}
\newcommand{\dert}{\partial_t}

\newcommand{\thetah}{\hat\thet}
\newcommand{\thetat}{\tilde\thet}

\newcommand{\vc}[1]{{\bf #1}}
\newcommand{\vv}{\vc{v}}
\newcommand{\vn}{\vc{n}}
\newcommand{\vw}{\vc{w}}
\newcommand{\thet}{\vartheta}

\renewcommand{\S}{\mathbf{S}}
\newcommand{\D}{\mathbf{D}}

\newcommand{\dt}{\,{\rm d} t }
\newcommand{\dd}{\,{\rm d} }
\newcommand{\dx}{\,{\rm d} {x}}
\newcommand{\dxdt}{\dx  \dt}

\newcommand{\dtau}{\,{\rm d} \tau}

\newcommand{\G}{\mathcal{G}}
\newcommand{\al}{\alpha}

\newcommand{\etan}{\eta_n}
\newcommand{\QT}{Q_T}
\newcommand{\intT}[1]{\int_0^T #1  \dt}
\newcommand{\intnn}[1]{\int_0^{\infty} #1  \dt}
\newcommand{\intnO}[1]{\int_0^{\infty}\intO{ #1 }\dt}
\newcommand{\intO}[1]{\int_{\Omega} #1 \dx}
\newcommand{\intTO}[1]{\int_0^T \int_{\Omega} #1  \dxdt}
\newcommand{\inttO}[1]{\int_0^{\tau} \int_{\Omega} #1 \dxdt}

\newcommand{\ff}{f}

\newcommand{\nm}[1]{\|#1\|}
\newcommand{\T}{\mathcal{T}}
\newcommand{\tmin}{\underline{\thet}}
\newcommand{\tmax}{\overline{\thet}}
\newcommand{\tstar}{t^*}
\newcommand{\bb}{b}
\newcommand{\B}{B}

\newcommand{\blafnc}{correction function}

\newcommand{\bws}{$\bb$-weak solution}
\newcommand{\Lab}{L_{\alpha,\beta}}

\newcommand{\dist}{\operatorname{dist}}

\newcommand{\cF}{\mathcal{F}}
\newcommand{\cT}{\mathcal{T}}

\newcommand{\pk}[1]{{\color{red}#1}}

\begin{document}

\title[Stability of a non-Newtonian Navier-Stokes-Fourier system]{On the stability of solutions to  non-Newtonian Navier--Stokes--Fourier-like systems in the supercritical case}
\author[A.~Abbatiello]{Anna Abbatiello}
\address{Anna Abbatiello.
University of Campania “L. Vanvitelli”, Department of Mathematics and Physics, Viale A. Lincoln 5, 81100 Caserta, Italy. ORCID: https://orcid.org/0000-0001-5758-114X}   \email{\tt anna.abbatiello@unicampania.it}

\author[M.~Bul\'{\i}\v{c}ek]{Miroslav Bul\'{\i}\v{c}ek}
\address{Miroslav Bul\'{\i}\v{c}ek. Charles University,  Faculty of Mathematics and Physics, Mathematical Institute, Sokolovsk\'{a} 83, 18675 Prague, Czech Republic. ORCID: https://orcid.org/0000-0003-2380-3458}
\email{\tt mbul8060@karlin.mff.cuni.cz}

\author[P.~Kaplick\'{y}]{Petr Kaplick\'{y}}
\address{Petr Kaplick\'{y}. Charles University, Faculty of Mathematics and Physics, Department of Mathematical Analysis, Sokolovsk\'{a} 83, 18675 Prague, Czech Republic. ORCID: https://orcid.org/0000-0002-1997-859X
}
\email{\tt kaplicky@karlin.mff.cuni.cz}


\thanks{The research activity of A.~Abbatiello is performed under the auspices of the Italian National Group for the Mathematical Physics (GNFM) of INdAM. M. Bul\'{\i}\v{c}ek and P. Kaplick\'{y} acknowledge the support of the projects  No. 20-11027X and No. 25-16592S financed by Czech Science Foundation (GA\v{C}R). M. Bul\'{\i}\v{c}ek is a member of the Ne\v{c}as Center for Mathematical Modeling.}

\begin{abstract}
We consider a three-dimensional domain occupied by a homogeneous, incompressible, non-Newtonian, heat-conducting fluid with prescribed nonuniform temperature on the boundary and no-slip boundary conditions for the velocity. No external body forces are assumed. The constitutive relation for the Cauchy stress tensor is assumed in a general form that includes, in particular, the power-law and Ladyzhenskaya models with the power-law exponent in the range where neither regularity, uniqueness, nor the validity of the energy equality is known to hold.

Nevertheless, we introduce a novel concept of solution suitable for this setting, which enables us to establish the existence of global-in-time solutions for arbitrary physically relevant initial data. A remarkable feature of this formulation is that the steady-state solution is nonlinearly stable: every such solution converges, in a suitable sense, to the steady state as time tends to infinity. This provides the first result that combines existence with long-time stability in this physically relevant yet mathematically challenging regime.
\end{abstract}

\maketitle

\section{Introduction}
We are interested in a modified Navier--Stokes--Fourier system with nonlinear viscosity
\begin{align}
\dert \vv + \diver(\vv \otimes \vv) -\diver \S &= - \nabla p,\label{MB1}\\
\diver \vv &= 0,\label{MB2}\\
\dert \thet + \diver ( \thet \vv) - \diver (\kappa(\thet)\nabla \thet) &= \S : \D \vv \label{MB3}
\end{align}
in $Q:=(0,\infty)\times \Omega$ with Lipschitz domain $\Omega\subset\R^3$. The system \eqref{MB1}--\eqref{MB3} is completed by the initial conditions
\begin{equation}\label{MB4}
\vv(0, \cdot)=\vv_0, \quad \thet(0, \cdot)=\thet_0 \qquad \textrm{ in } \Omega
\end{equation}
and by the boundary conditions
\begin{equation}
\label{MB5}
\vv=\vc{0}, \quad \thet=\thetah   \qquad \textrm{ on } (0, \infty)\times \partial \Omega.
\end{equation}
Here, $\vv$ is the velocity field, $p$ is the pressure and $\thet$ is the temperature. The material is described by the heat conductivity $\kappa:\mathbb{R}_+ \to \mathbb{R}_+$, and by the heat capacity. For simplicity and for the clarity of the presentation we consider the heat capacity being equal to one. Further, we need to specify the constitutively determined part of the Cauchy stress $\S$, that is supposed to be a tensorial function of the temperature $\thet$ and the symmetric velocity gradient $\D\vv:=(\nabla \vv +(\nabla \vv)^T)/2$, i.e., there exists $\S^*:\mathbb{R}_+ \times \mathbb{R}^{3\times 3}_{\textrm{sym}} \to \mathbb{R}^{3\times 3}_{\textrm{sym}}$ such that 
\begin{equation}\label{def:S}
\S = \S^*(\thet, \D \vv) \qquad \textrm{ in }Q.
\end{equation}
We assume that the mapping $\S^*:(0,+\infty)\times\R^{3\times 3}_{sym}\to \R^{3\times 3}_{sym}$ is continuous and  fulfills the following growth hypothesis
with some positive constants $\underline{\kappa}, \overline{\kappa}, \delta >0$ 
and $p\in [8/5, 11/5)$
\begin{equation}
\label{MBBounds}
\begin{gathered}
\underline{\kappa}\le \kappa(\thet)\le \overline{\kappa}, \qquad |\S^*(\thet,\D)|\le \overline{\kappa}(\delta + |\D|)^{p-1}, \qquad |\S^*(\thet, \vc{0})|=0,\\
\underline{\kappa}(\delta + |\D|)^{p-2}|\D|^2 \le \S^*(\thet,\D):\D.
\end{gathered}
\end{equation}
For the proof of existence of a solution we also need that $\S^*$ is monotone with respect to the second variable, i.e., that  for all $\thet>0$, $\D_1,\D_2\in\R^{3\times 3}_{sym}$, $\D_1\neq\D_2$ there holds
\begin{equation}\label{s:monotone}
(\S^*(\thet,\D_1)-\S^*(\thet,\D_2)):(\D_1-\D_2)\ge 0.
\end{equation}
Note that the prototypic relation $\S^*(\thet,\D)=\nu(\thet)(\delta + |\D|)^{p-2}\D$ with smooth bounded positive function $\nu$ bounded away from zero falls into the class \eqref{MBBounds} and \eqref{s:monotone}.

To simplify the notation, we consider for the \emph{stationary} boundary data $\thetah$ that it is extended into $\Omega$ such that the extension is the expected asymptotic solution, i.e., we assume that for some constants $\tmin,\tmax>0$
\begin{equation}\label{MB6}
\thetah\in W^{1,2}(\Omega),\quad 0<\tmin\le \thetah \le \tmax, \quad \mbox{ and for all } \varphi\in W^{1,2}_0(\Omega)  \quad \intO{\kappa(\thetah)\nabla \thetah\cdot \nabla \varphi}=0.
\end{equation}
The initial conditions are considered to have finite energy, i.e.,
\begin{equation}\label{MB7}
\intO{\frac{|\vv_0|^2}{2} + \thet_0} <\infty
\end{equation}
and to satisfy the standard compatibility conditions
\begin{equation}\label{MB7.5}
\diver \vv_0 =0, \; \thet_0 >\tmin \textrm{ in } \Omega \qquad \textrm{ and } \qquad \vv_0  \cdot \vc{n} = 0 \textrm{ on } \partial\Omega,
\end{equation}
where $\vc{n}$ denotes the unit outward normal vector on $\partial \Omega$.

It is seen that $(\vc{0}, \thetah)$ 
is the only steady state of the problem \eqref{MB1}--\eqref{MB5}. We want to study its nonlinear stability, i.e., we are interested under which conditions and for which class of solutions this steady state \emph{attracts} in some sense all such solutions. Furthermore, we want to be sure that the class of relevant solutions is not empty.

As it is common in the theory of complicated systems arising in the theory for the heat conducting fluids, see \cite{A21, ABK22, AbBuKa2024, FeNo2021, FeiMa2006, FeNo2017}, distributional solutions need not be sufficiently regular to guarantee expected properties such as the stability of steady states. For this reason, additional conditions must be incorporated into the definition of a solution. These are typically formulated as inequalities that capture the behavior of key physical quantities such as the total energy, the internal energy, the kinetic energy and the entropy. Note that in the case of the Navier--Stokes system (the system without the temperature effects), such a condition is the classical (kinetic) \emph{energy inequality}. In the present work, we seek suitable variants of such constraints to develop a complete stability theory for the steady-state regime. Specifically, we replace the heat equation with the entropy inequality, supplemented by a modified version of the total energy inequality. This latter inequality is nonstandard, as it includes a correction term that accounts for the heat flux through the boundary. This modification is necessary because we impose a nonhomogeneous Dirichlet boundary condition on the temperature, implying that the system is not isolated and may exchange energy with its surroundings.
The precise definition of a weak solution is presented later. We show the global-in-time and large-data existence of such solutions and prove that they all converge to the unique steady state as time $t\to +\infty$.

Very similar problems were studied in \cite{BuMaPr19, DosPruRaj, ABK22, AbBuKa2024}. However, the situations considered there were significantly simpler: either the existence of a hypothetical smooth solution was assumed, or the results were restricted to the so-called subcritical regimes. In such settings, one may directly apply energy equalities and classical renormalization techniques, and the internal energy equality is equivalent to the entropy equality. Furthermore, for given temperature, we have the uniqueness of the velocity field, see \cite{BuKaPr19}. Nonlinear stability of the steady state can then be deduced using these tools. In contrast, the methods and techniques developed in the aforementioned works are not applicable in the supercritical setting, which is, in fact, the most natural regime in the context of incompressible fluids. For instance, the classical three-dimensional Navier–-Stokes-–Fourier system falls outside the scope of the results presented in these papers.

As the main contribution of the present work, we introduce a new notion of solution (see Definition~\ref{def:blafnc}), for which we prove global-in-time existence for arbitrary initial data with finite energy (Theorem~\ref{thm:ex}). Most importantly, we establish the nonlinear stability of the steady state for this system (Theorem~\ref{TMB}). It is remarkable that the new concept of a solution introduced in Definition~\ref{def:blafnc} can, in a certain sense, also be identified with the classical notion of suitable weak solutions introduced in \cite{CaKoNi82}.

Furthermore, we observe that some additional property must be required of the solution in order to guarantee stability, since it is known that weak solutions, or even Leray--Hopf solutions, need not be unique for shear-thinning fluids, see \cite{BuMoSz21}. In contrast, the authors of \cite{AbFe2020} introduced another general concept of solution for shear-thinning fluids, for which they also established an existence result. Moreover, they obtained a conditional uniqueness result, which may be understood as a first step toward stability. However, all of this holds under the hypothesis that a sufficiently regular solution exists. Such a result in the three-dimensional setting is not generally known; at present we only have small-data regularity results for shear-thinning fluids, see \cite{AM19, BeDieRu2010}.

A related stability result was also recently established in \cite{FeLuSu2024}  (see \cite{FeNo2022-kniha} for the existence results), where the authors study an even more complex problem involving a compressible fluid. Their model allows not only the transfer of kinetic energy into heat, but also the reverse process. In their case, the steady state is not necessarily unique in general. Therefore, they assume that the boundary data are close to a constant steady solution, which, together with the smallness of the potential driving force, guarantees the uniqueness of the steady state. Another key difference compared to our work is that the authors consider a more restrictive class of weak solutions. Specifically, they require that the solution satisfies a relative energy inequality directly, which is not the case of this paper. 
We also note that a related problem for the compressible magnetohydrodynamics (MHD) system was recently initiated in~\cite{FeGwiaKwonSwie2023}.


Our system may also be compared to the Oberbeck--Boussinesq approximation; however, it differs from it in essential aspects. The Oberbeck--Boussinesq system does not account for the dissipation of kinetic energy in the heat equation, which is the main difficulty in our analysis. On the other hand, it includes the buoyancy force in the balance of linear momentum, which is the term responsible for transferring heat into kinetic energy. At present, we are not able to handle this term.

A variant of the Oberbeck--Boussinesq system with nonhomogeneous Dirichlet boundary conditions was derived in~\cite{BeFeOsch2023} by a singular limit passage from the full compressible system. Surprisingly, a new boundary condition appeared, involving a nonlocal boundary term. The existence of solutions to this new system was established in~\cite{AbFe2023}, and nonlinear stability for small data was proved in~\cite{FeRoSchi2024}. We emphasize that the analysis in these works differs essentially from ours, since the inclusion of dissipative heating in the heat equation introduces many additional difficulties. In particular, the heat equation must be treated as an equation with an $L^1$ right-hand side.
In this context, existence results for the Oberbeck--Boussinesq approximation with additional dissipative heating are already available: see \cite{KaRuTha2000} for the derivation of the model and \cite{AbBuLe24} for the corresponding existence result. However, for such a complicated system no further analysis is available to date, and we plan to address this in forthcoming papers.



\bigskip
  
The rest of the introductory part is given to the introduction of the notation we use and also for the main result we get.

\subsection{Functional setting}

We assume that $\Omega \subset \mathbb{R}^3$ is a bounded domain with $\mathcal{C}^{0,1}$ boundary.  
For a fixed $T > 0$, we use the notation $\QT = (0,T)\times \Omega$, while for $T = \infty$ we write $Q = (0,+\infty)\times \Omega$.  

The scalar product of vectors in $\mathbb{R}^3$ is denoted by $\cdot$, and for the scalar product of matrices we use the symbol $:$.   The set of symmetric $3 \times 3$ matrices is denoted by $\mathbb{R}^{3 \times 3}_{\mathrm{sym}}$.  
Differential operators $\diver$, $\Delta$, $\nabla$ always refer to the spatial variable $x \in \Omega$, while $\partial_t$ denotes the partial derivative with respect to $t \in (0,T)$ or $(0,+\infty)$.  We denote by $|\Omega|$ the Lebesgue measure of $\Omega$, and by $\chi_\Omega$ its indicator function. 

We use standard notation for Lebesgue, Sobolev, and Bochner spaces and their norms, e.g. 
$L^p(\Omega)$, $W^{1,p}(\Omega)$, $L^p(0,T; X)$, $W^{1,p}(0,T; X)$ for $p\in [1,\infty]$, where $X$ is a Banach space. 
If we want to emphasize that functions are vector-valued, we write, for instance, 
$L^p(\Omega;\mathbb{R}^3)$.  Smooth functions with compact support are denoted by $\mathcal{C}^\infty_0(\Omega)$ or 
$\mathcal{C}^\infty_0(0,+\infty)$.  
The restrictions of smooth functions to a given closed set are denoted again, e.g. by 
$\mathcal{C}^\infty(\overline{\Omega})$. Furthermore, the symbol $\mathcal{C}^k$ for $k \in \mathbb{N}$ denotes the space of 
$k$-times continuously differentiable functions, while for continuous functions we simply 
write $\mathcal{C}$.

The spaces $L^2_{0,\diver}$ and $W^{1,p}_{0,\diver}$ for $p\in [1,\infty)$ are defined as the closures of 
\[
\{\varphi \in \mathcal{C}^\infty_0(\Omega) \; ; \; \diver \varphi = 0\}
\]
in $L^2(\Omega)$ and $W^{1,p}(\Omega)$, respectively. Similarly, the space  $W^{1,p}_0(\Omega)$ denotes the classical zero trace Sobolev space.  
On these spaces we use the inherited norms.  
The dual space of a Banach space $X$ is denoted by $X^*$, and the duality pairing by $\langle \cdot, \cdot \rangle_X$.  
If it is clear from the context, the subscript $X$ will be omitted.  
We also introduce the space $\mathcal{M}(0,T; X^*)$ of $X^*$-valued countably additive regular measures with bounded variation, which is isometrically isomophic to $\mathcal{C}([0,T],X)^*$, see~\cite{DieSwart2002}.

Moreover, we will work with functions that are weakly continuous in time.  
We denote by $\mathcal{C}_{\mathrm{weak}}(0,T;X)$ the space
\[
\mathcal{C}_{\mathrm{weak}}(0,T;X) = 
\Bigl\{ u \in L^\infty(0,T;X);
\forall \varphi \in X^*: \langle \varphi, u(\cdot) \rangle_X \in \mathcal{C}[0,T] \Bigr\},
\]
where $X$ is a Banach space.

\newcommand{\cad}{c\`{a}dl\`{a}g } 

Finally, since we are unable to address the continuity of the temperature with respect to the time variable, we employ the notion of \cad functions and refer the reader to \cite{AbLiSc24,Am90} for more details. In our setting, we only need the following properties. We define the space of $X$-valued \cad functions as
$$
\mathcal{D}(0,T; X):=\{u\in L^{\infty}(0,T; X); (\forall t\in [0,T): u(t) = \lim_{s\to t_+} u(s))\land (\forall t\in (0,T]:\exists \lim_{s\to t_{-}}u(s)\in X)\},
$$
which means that these functions are right-continuous and have left limits with respect to the norm topology in $X$. Further, we define weak \cad functions as
$$
\mathcal{D}_{\textrm{weak}}(0,T; X):=\{u\in L^{\infty}(0,T; X); \forall \varphi \in X^*: \langle \varphi,u(\cdot) \rangle_X \in \mathcal{D}(0,T; \R) \}.
$$
Finally, we state the following properties. First, if $u\in L^1(0,T; X)$ and $\partial_t u\in \mathcal{M}(0,T; X)$, then $u\in BV(0,T; X)$ 
and there exists a representative, denoted again by $u$, such that $u \in \mathcal{D}(0,T; X)$; see \cite[Lemma~2.5]{AbLiSc24}. Further, assume that the Banach space $X$ is reflexive and is densely and continuously embedded into $Y$. Then
$$
L^{\infty}(0,T; X) \cap \mathcal{D}_{\textrm{weak}}(0,T; Y) \subset \mathcal{D}_{\textrm{weak}}(0,T; X).
$$
This can be proved almost line by line as in \cite[Lemma~III.1.4]{temam79}, which asserts
$$
L^{\infty}(0,T; X) \cap \mathcal{C}_{\mathrm{weak}}(0,T; Y) \subset \mathcal{C}_{\mathrm{weak}}(0,T; X).
$$
Consequently, any function $u \in L^1(0,T; X)$ with $\partial_t u \in \mathcal{M}(0,T; Y)$, where $X$ is reflexive and densely and continuously embedded into $Y$, has a representative $u \in \mathcal{D}_{\textrm{weak}}(0,T; X)$.

Throughout the paper we use a generic constant $C>0$, whose value may change from line to line but which is always independent of the relevant quantities.  

\subsection{Definition of the notion of solution}

Let us now turn to the definition of our concept of solution. Since we aim to address the so-called supercritical case, we cannot expect a solution that satisfies \eqref{MB3}. We are inspired by similar approaches for compressible fluids \cite{Fe04} and also in the incompressible setting, see \cite{BuFeMa2009,BuMaRa09}. Instead of the internal energy identity \eqref{MB3}, we consider the entropy ``equality'', which is derived by dividing \eqref{MB3} by $\thet$. Defining the entropy $\eta$ as
$$
\eta := \ln \thet,
$$
this procedure leads to
\begin{equation}
\dert \eta + \diver (\eta \vv) - \diver (\kappa(\thet)\nabla \eta) = \frac{\S : \D \vv}{\thet} + \kappa(\thet)|\nabla \eta|^2.
\label{MB34}
\end{equation}
Similarly, we define the entropy associated with the stationary temperature $\hat{\thet}$ as $\hat{\eta} = \ln \hat{\thet}$. It is important to note that $\eta$ has much better integrability properties than $\thet$, which is the main reason for working with it. Nevertheless, we are still unable to handle the equality in \eqref{MB34} and therefore we replace it by an inequality ``$\ge$", which is natural from the perspective of the second law of thermodynamics. 

However, in doing so we lose some information, and the system \eqref{MB1}--\eqref{MB2} together with the inequality \eqref{MB34} does not constitute a complete problem. Therefore, an additional restriction is required in order to formulate a well-posed problem. One natural possibility is to introduce the kinetic and total energy (in)equalities.

The kinetic energy identity is obtained by taking the scalar product of \eqref{MB1} with $\vv$, which yields
\begin{equation}\label{MBkinetic}
  \partial_t \left(\frac{|\vv|^2}{2}\right) 
  + \diver \left(\vv \left(p+\frac{|\vv|^2}{2}\right)\right) 
  - \diver (\S \vv) 
  + \S : \D\vv = 0.
\end{equation}
Furthermore, by adding \eqref{MBkinetic} to \eqref{MB3} we obtain the total energy balance
\begin{equation}\label{MBtotal}
  \partial_t \left(\thet + \frac{|\vv|^2}{2}\right) 
  + \diver \left(\vv \left(\thet + p+\frac{|\vv|^2}{2}\right)\right) 
  - \diver (\S \vv) 
  - \diver (\kappa(\thet)\nabla \thet) = 0.
\end{equation}

The classical approach is to integrate the above identity over $\Omega$ and to incorporate it, together with \eqref{MB1}--\eqref{MB2} and the inequality \eqref{MB34}, as part of the solution concept. This indeed leads to a well-defined problem in the case of Neumann boundary conditions for $\thet$, see \cite{Fe04,BuFeMa2009,BuMaRa09}. 
However, for Dirichlet boundary conditions such a procedure is not possible. After integration over $\Omega$ and applying integration by parts, one obtains the additional boundary term
\[
  \int_{\partial \Omega} \kappa(\thet)\nabla \thet \cdot \vn \, \mathrm{d}S,
\]
which cannot be controlled. Therefore, \eqref{MBtotal} must be properly modified, and a different concept needs to be introduced. To this end, we first introduce the notion of the \blafnc.
\begin{definition}[$\gamma$-\blafnc]\label{def:blafnc}
Let $\gamma \in (-1,0)$. We say that 
$\bb \in \mathcal{C}^1((0,+\infty) \times [\tmin,\tmax])$ 
is a $\gamma$-\blafnc\ if there exists $C>0$ such that $\bb$ and its primitive function with respect to the first variable
\begin{equation}\label{Bdef}
\B(s, \tilde{s}) = \int_{\tmin}^s \bb(\sigma, \tilde{s}) \, {\rm d} \sigma, 
\quad \text{for } s \in (0,+\infty) \text{ and } \tilde{s} \in [\tmin,\tmax],
\end{equation}
satisfy, for $s \geq \tmin$ and $\tilde{s} \in [\tmin,\tmax]$,
\begin{align}
\label{item0}  
0 \leq b(s,\tilde s) &\leq C s^\gamma, & b(\tilde{s},\tilde{s}) &= 1, \\
\label{item2} 
|\partial_2 b(s, \tilde{s})| &\leq C s^\gamma, & |\partial_2 B(s, \tilde{s})| &\leq C s^{1+\gamma},\\
\label{item3}
- C s^{\gamma-1} \leq \partial_1 b(s, \tilde{s}) &\leq 0.
\end{align}
\end{definition}
Note that it follows directly from the above definition that there exists a constant $C>0$ such that, for all $s \geq \tmin$ and all $\tilde{s} \in [\tmin,\tmax]$, 
\begin{equation}\label{rem:B-growth}
  |B(s,\tilde s)| \leq C s^{1+\gamma}.
\end{equation}
The typical example of the $\gamma$-\blafnc\ is the following function
\begin{equation}\label{bEX}
  \bb(\thet,\thetah)=\left(\frac \thetah\thet\right)^\alpha,\quad B(\thet,\thetah)=\frac{\thetah^\alpha}{1-\alpha}\left(\thet^{1-\alpha}-\tmin^{1-\alpha}\right),
\end{equation}
for some $\alpha\in(0,1)$. In this setting we have that $\bb$ satisfies Definition~\ref{def:blafnc} with $\gamma=-\alpha$.

The main motivation for introducing the $\gamma$-\blafnc\ is the following. If we multiply~\eqref{MB3} by $b(\thet,\hat{\thet})$ and use the fact that 
$\hat{\thet}$ is time-independent, then
\begin{equation}\label{ABC}
\begin{split}
&\S : \D \vv \, b(\thet,\hat{\thet}) = \dert  \thet\,  b(\thet,\hat{\thet}) + \vv\cdot \nabla \thet \, b(\thet,\hat{\thet}) - b(\thet,\hat{\thet})\, \diver (\kappa(\thet)\nabla \thet) \\
&\quad = \dert  B(\thet,\hat{\thet}) + \vv\cdot \nabla \thet \, b(\thet,\hat{\thet}) - \diver \left(\kappa(\thet)\nabla \thet\, b(\thet,\hat{\thet})\right)+\, \kappa(\thet)\nabla \thet \cdot \nabla \big(b(\thet,\hat{\thet})).
\end{split}
\end{equation}
Finally, let $\varphi \in \mathcal{C}^{1}(\overline{\Omega})$ be a nonnegative function such that $\varphi \equiv 1$ on $\partial \Omega$. Multiplying \eqref{ABC} by $\varphi$, adding the result to \eqref{MBtotal}, and performing a straightforward algebraic manipulation, we obtain
\begin{equation}\label{MBtotal2}
\begin{split}
  &\partial_t \left(\thet -\varphi  B(\thet,\hat{\thet}) + \frac{|\vv|^2}{2}\right) + \left(\S : \D \vv \, b(\thet,\hat{\thet}) -\, \kappa(\thet)|\nabla \thet|^2 \partial_1 b(\thet,\hat{\thet}) \right)\varphi=\\
  &\quad
  - \diver \left(\vv \left(\thet + p+\frac{|\vv|^2}{2}\right)\right) 
  + \diver (\S \vv) 
  + \diver \left(\kappa(\thet)\nabla \thet (1-\varphi b(\thet,\hat{\thet}))\right) \\
     &\quad + \kappa(\thet)\nabla \thet \cdot \nabla \varphi \, b(\thet,\hat{\thet})+\, \kappa(\thet)\nabla \thet \cdot \nabla \hat{\thet} \partial_2 b(\thet,\hat{\thet}) \varphi +\vv\cdot \nabla \thet \, b(\thet,\hat{\thet})\varphi
\end{split}
\end{equation}
The above identity is a natural candidate for the final restriction on the solution. First, using \eqref{MBBounds}, \eqref{item0}, \eqref{item3}, and the fact that $\varphi \ge 0$, one observes that the last term on the left-hand side is nonnegative. Second, since $\vv = \vc{0}$, $\varphi = 1$, and $b(\thet,\hat{\thet})=1$ on $\partial \Omega$, we may integrate~\eqref{MBtotal2} over~$\Omega$. In this case, all divergence terms vanish by the integration by parts formula, since their traces are zero. This is the crucial difference compared to \eqref{MBtotal}. From these manipulations we deduce the following \emph{ $b$ - corrected total energy inequality} (after discarding terms with the correct sign)
\begin{equation}\label{MBtotal3}
\begin{split}
  &\partial_t \left(\intO{\thet -\varphi  B(\thet,\hat{\thet}) + \frac{|\vv|^2}{2}}\right) \le \\
     &\quad +\intO{\kappa(\thet)\nabla \thet \cdot \nabla \varphi \, b(\thet,\hat{\thet})+\, \kappa(\thet)\nabla \thet \cdot \nabla \hat{\thet} \partial_2 b(\thet,\hat{\thet}) \varphi +\vv\cdot \nabla \thet \, b(\thet,\hat{\thet})\varphi}
\end{split}
\end{equation}
and this is the key point for the definition of a proper notion of solution. Thus, we complement \eqref{MB1}--\eqref{MB2} and the inequality \eqref{MB34} by additionally requiring \eqref{MBtotal3}.  
We are now in a position to introduce the rigorous notion of solution.
\begin{definition}[\bws]\label{proper-sol}
Let $\gamma \in (-1,0)$ and let $\bb$ be a $\gamma$-\blafnc. Assume that $\Omega \subset \mathbb{R}^3$ 
is a bounded Lipschitz domain and suppose that $\S^*$ satisfies the conditions~\eqref{MBBounds}--\eqref{s:monotone} with $p>\frac65$,  
and that $\hat{\thet}$ fulfills \eqref{MB6} for some $0 < \tmin \leq \tmax$. Furthermore, let the initial data 
$\vv_0$ and $\thet_0$ be given such that 
\[
  \vv_0 \in L^2_{0,\diver}(\Omega;\mathbb{R}^3), \qquad 
  \thet_0 \in L^1(\Omega), \qquad 
  \thet_0 \geq \tmin \ \text{a.e. in } \Omega,
\]
that is, assumptions \eqref{MB7}--\eqref{MB7.5} are satisfied.  

We say that a quadruple $(\vv,\S,\thet,\eta)$ is a \bws\ provided that for every $T > 0$ the following conditions hold:
\begin{itemize}[leftmargin=*]
\item we have the following point-wise relations
$$
\S=\S^*(\thet,\D\vv), \quad \eta=\ln \thet, \quad \thet\ge \tmin \qquad \textrm{a.e. in } Q_T;
$$
\item we have the following functional setting
\begin{align}\label{MBap}
\begin{aligned}
&\vv \in L^\infty(0,T; L^2_{0,\diver}) \cap L^p(0,T; W^{1,p}_{0,\diver}),\quad \partial_t\vv \cap L^{5p/6}\left(0,T; \left(W^{1,(5p/6)'}_{0,\diver}\right)^*\right),\\
&\S \in L^{p'}((0,T)\times \Omega),\qquad \thet \in L^\infty(0,T; L^1(\Omega)), \qquad \partial_t \thet \in \mathcal{M}\left(0,T; (W^{2,3}_0(\Omega))^*\right),\\
&\eta-\hat{\eta} \in L^2(0,T; W^{1,2}_0(\Omega)), \qquad \partial_t \eta \in \mathcal{M}\left(0,T; (W^{2,3}_0(\Omega))^*\right);
\end{aligned}
\end{align}
\item the initial data are attained in the following sense
\begin{equation}\label{eq:att-ic}
\lim_{t\to 0_+} \|\vv(t)-\vv_0\|_2^2+ \|\thet(t)-\thet_0\|_1  = 0;
\end{equation}
\item the balance of linear momentum~\eqref{MB1} is satisfied in the following sense: for all $\vw\in W^{1,(5p/6)'}_{0,\diver}$ and for almost all~$t\in (0,T)$ 
\begin{equation}\label{weak1}
\langle \partial_t\vv, \vw \rangle + \intO{(\S- \vv \otimes \vv): \nabla \vw}=0; 
\end{equation}
\item  the entropy inequality \eqref{MB34} is satisfied in the following sense: for every nonnegative $\varphi \in \mathcal{C}^1_0(\Omega)$ and nonnegative $\psi \in \mathcal{C}_0^1 (0,T)$
\begin{equation}\label{entropy-limit}
\begin{split}
&-\intTO{\eta \varphi \dert\psi}  - \intTO{\eta\,\vv\cdot \nabla\varphi \psi}  + \intTO{\kappa(\thet)\nabla\eta\cdot\nabla\varphi \psi} \\
&\qquad \geq \intTO{\frac{\S:\D{\vv}}{\thet}\,\varphi \psi} + \intTO{\kappa(\thet)\,|\nabla \eta|^2\, \varphi \psi}; 
\end{split}
\end{equation}

\item the kinetic energy inequality \eqref{MBkinetic} holds in the following sense:  for all nonnegative $\psi\in \mathcal{C}_0^1(0,T)$ 
  \begin{equation}\label{eq:kinen1}
\intTO{-\frac{|\vv|^2}{2}\, \dert\psi + \S:\D\vv \, \psi }\leq 0;
  \end{equation}
\item the $b$ - corrected total energy inequality \eqref{MBtotal3}  holds in the following sense: for any nonnegative $\varphi\in \mathcal{C}^1(\overline{\Omega})$ fulfilling  $\varphi=1$ on $\partial \Omega$ and for any nonnegative  $\psi\in \mathcal{C}^\infty_0(0, T)$    
\begin{equation}\label{TE}
\begin{aligned}
  &- \intTO{\left[\frac{|\vv|^2}{2} + \thet - \B(\thet, \thetah)\,\varphi\right] \,\partial_t \psi}
  -\intTO{\vv\cdot\nabla \thet\, \bb(\thet, \thetah)\varphi\,\psi}\\
 &\quad - \intTO{\kappa(\thet)\nabla\thet\cdot \nabla \varphi\, \bb(\thet, \thetah)\,\psi}
  -\intTO{\kappa(\thet)\nabla\thet\cdot\nabla\thetah \,\partial_{2} \bb(\thet, \thetah)\,\varphi \,\psi}
  \leq 0.
\end{aligned}
\end{equation}
\end{itemize}
\end{definition}

\bigskip

Before stating the main results of the paper, we briefly discuss the definition introduced above. First, it follows from the definition that $\eta\in L^\infty(0,T; L^q(\Omega))$. Next, it is important to emphasize that we always choose the representatives of $\vv$ and $\thet$ in such a way that $\vv \in \mathcal{C}_{\textrm{weak}}(0,T; L^2(\Omega))$ and $\eta\in \mathcal{D}_{\textrm{weak}}(0,T; L^2(\Omega))$. Thus, $\vv$, $\eta$, and consequently also $\thet$ are defined for all $t\in (0,T)$.

In Definition~\ref{proper-sol}, we implicitly assume that all integrals and expressions are well defined. Note that the regularity assumption \eqref{MBap} is not sufficient to guarantee that all integrals in the definition are well defined for every $\gamma \in (-1,0)$. In particular, to ensure that the third and fourth integrals in \eqref{TE} are well defined for $\thetah \in W^{1,2}(\Omega)$, it is necessary that
\begin{equation}\label{ADD2}
|\nabla \thet| \, \thet^{\gamma} \in L^1(0,T; L^2(\Omega)),
\end{equation}
by \eqref{item0}--\eqref{item2}, which can be derived from \eqref{MBap} only under certain restrictions on the value of $\gamma$. Henceforth, $\gamma \in (-1,0)$ will always be chosen so that every integral in the above definition is well defined and finite.


\subsection{Main results}
The first result concerns the existence of global-in-time and large-data \bws\ for arbitrary $\bb$ with a suitable $\gamma$.
\begin{theorem}\label{thm:ex}
Let $\gamma\in(-1,-1/2)$ and let $\bb$ be a fixed $\gamma$-\blafnc. Assume that the data $\thetah$ satisfy \eqref{MB6}, and the initial data $\vv_0$ and $\thet_0$ satisfy \eqref{MB7} and \eqref{MB7.5}, respectively. Assume further that $\kappa$ and $\S^*$ satisfy \eqref{MBBounds}--\eqref{s:monotone} with $p\in(6/5,11/5)$.
Then there exists a \bws\ to the problem \eqref{MB1}-\eqref{MB5}. Moreover, the solution satisfies the global estimate
\begin{equation}\label{eq:lnl1}
\nm{\thet}_{L^\infty(0,+\infty,L^1(\Omega))}\leq \intO{\frac12|\vv|^2+\max(\tmax,\thet_0)}.
\end{equation}
\end{theorem}
The key novelties of this result are several. First, we are able to prove the existence for parameters $p>6/5$, whereas previous results dealt with larger exponents $p>9/5$ and only the Neumann boundary data for temperature; see \cite{BuMaRa09}. This is because in this paper, we use a different concept of the solution, namely a solution fulfilling the $b$-corrected total energy inequality \eqref{MBtotal3}, which allows us to go beyond the threshold of $9/5$ in \cite{BuMaRa09}. Next, the upper bound $p<11/5$ can in fact be omitted, as the proof is much easier. For such problems, we refer to \cite{ABK22}. Finally, and importantly, for the stability result, we have the uniform estimate \eqref{eq:lnl1}. We would also like to mention that in the case where $\hat{\thet}\in W^{1,\infty}(\Omega)$, we can extend the result for all $\gamma\in (-1,0)$. Since such an extension is very straightforward, we do not provide the details in this paper.

\bigskip

For our second result (the stability of the steady state) we first introduce auxiliary functions defined for $s,t>0$ as
\begin{equation}\label{DFG}
\mathcal{H}^\al(s):= \int_1^s (\G(z))^{-\al}\,{\rm d}z,\quad \G(s):=\int_0^s \kappa(z)\, {\rm d} z,\quad \ff_\alpha(s,t):=s-t - (\mathcal{H}^\al(s) - \mathcal{H}^\al(t))(\G(t))^{\al}.
\end{equation}
These functions are necessary due to the nonconstant heat capacity $\kappa$. In the case where $\kappa$ is constant, we could use the $\gamma$-\blafnc \ directly, for example, the one given in \eqref{bEX}. Since $\kappa$ is nonconstant, we need to introduce \eqref{DFG}. As the pair $(\vc{0},\thetah)$ is the only steady state of the problem \eqref{MB1}--\eqref{MB5}, we want to show that any \bws\ converges in some sense to $(\vc{0},\thetah)$ as $t\to \infty$. For this purpose, we need to introduce a proper concept for measuring the distance between functions - a certain Lyapunov functional. Inspired by recent works \cite{AbBuKa2024, DosPruRaj}, we define a class of Lyapunov functionals $\Lab:\mathbb{R}^3\times(0,+\infty)\times(0,+\infty)\to\mathbb{R}$ for $\alpha\in(0,1)$ and $\beta>0$ by
\begin{equation}\label{LMB}
\Lab(\vv,\thet,\thetah):= \beta|\vv|^2 + \thet-\thetah - (\mathcal{H}^\al(\thet) - \mathcal{H}^\al(\thetah))(\G(\thetah))^{\al}.
\end{equation}
Here we have used the functions $\mathcal{H}^\al, \G:(0,+\infty)\to\R$ defined in \eqref{DFG}. Note that it follows directly from the definition that $\Lab$ is convex with respect to $\vv$ and $\thet$, that $\Lab$ is nonnegative, and that $\Lab(\vv, \thet, \thetah)=0$ if and only if $\vv =\vc{0}$ and $\thet=\thetah$.

\bigskip

Finally, we formulate the main result of the paper, which concerns the stability of the steady state for \bws\ measured by $\Lab$.
\begin{theorem}\label{TMB}
Let $p\in[6/5,11/5)$ and $\underline{\kappa}$, $\overline{\kappa}$, $\underline{\thet}$, $\overline{\thet}>0$ be given and fulfill $\tmin\leq\tmax$, $\underline{\kappa}\leq\overline{\kappa}$. Then there exists $\mu>0$ depending only on $\underline{\kappa}$ and $\Omega$ ($\mu$ is explicitly given below \eqref{eq:sdv}) such that any \bws\ satisfies  for all  $\tau> 0$
\begin{equation}\label{eq:twocrosses}
\|\vv(\tau)\|_2^2\le e^{-\mu \tau}\|\vv(0)\|_2^2,
\end{equation}
provided $p\geq 2$ or $p<2$ with $\|\vv(0)\|_2\leq 1$. In the case $p<2$ and $\|\vv(0)\|_2> 1$ we get 
\begin{equation}\label{snad}
\|\vv(\tau)\|_2^2\le e^{-\mu \tau}\|\vv(0)\|_2^{\frac2p}e^{\|\vv(0)\|_2^2-1}.
\end{equation}

Moreover, let $p \in [8/5,2]$ and $\al \in (1/2,2/3] \cap [\,2-5p/6,\,2/3]$, or let $p \in [2,11/5]$ and $\al \in (1/2,2/3]$.  
Let $R>0$ be arbitrary and fixed. Then there exist constants $\mu,\beta>0$ such that any $\bb$-weak solution with a $(-\alpha)$-\blafnc\ $b$ satisfies, for almost every $\sigma \geq 0$ and for all $\tau \geq \sigma$,
\begin{equation}
  \|\vv_0\|_2^2 + \|\thet_0\|_1\le R \implies
  \intO{L_{\alpha,\beta}(\vv(\tau),\thet(\tau),\thetah)} \le e^{-\mu(\tau-\sigma)}\intO{L_{\alpha,\beta}(\vv(\sigma),\thet(\sigma),\thetah)} . \label{assymptotictwo}
\end{equation}
\end{theorem}




Please note that the statement of the theorem is almost identical to the main theorem in~\cite{AbBuKa2024}.  
Nevertheless, the nature of the result is substantially different. In particular, the admissible range of the growth parameter $p$ is different, which necessitates the introduction of a notion of weak solution distinct from the one employed in~\cite{AbBuKa2024}.  
This modification is not merely of technical character, but rather reflects a deeper difference in the underlying analysis.  
As a consequence, several arguments used in~\cite{AbBuKa2024} cannot be applied directly in the present setting, and additional tools are required to establish the desired estimates.  
In particular, the adaptation of the entropy inequality plays a crucial role in our approach.  
Of course, one could attempt to construct solutions directly so that they satisfy the inequality \eqref{dec-temp}. However, our result goes significantly further, as it provides a much stronger conclusion than what such a direct construction would yield.

Also, note that the long time behaviour simplifies when $\delta=0$ in $\S$ and $p<2$, since then the extinction in a finite time of the velocity holds, see \cite{A21}. Thus $u$ decays to zero in finite time and  only the heat equation without the $L^1$ term remains for large times.


The article is split into two main parts. In Section~\ref{sec:ex} we prove existence of the \bws\ and in Section~\ref{sec:conv} we show that any \bws\ must converge to the unique steady state.



\section{Existence of \texorpdfstring{\bws}{b} - the proof of Theorem~\ref{thm:ex}}
\label{sec:ex}

In this section, we provide a complete proof of existence. As many steps parallel the procedure presented in \cite{ABK22}, we occasionally omit routine details and highlight instead the essential differences.

\subsection{Faedo--Galerkin approximations}
For any $n\in \N$ there exists a pair $(\vv^n, \thet^n)$ solving for almost all $t\in (0,+\infty)$ the following system
\begin{align}
\label{ode11}
&\intO{[\partial_t\vv^{n} \cdot \vw_j  - (\vv^{n} \otimes \vv^{n}): \nabla \vw_j + \S^{n}: \D\vw_j]} = 0 \qquad \mbox{ for all } j=1,\dots,n,
\\
\label{ode22}
&\begin{aligned}
&\dual{\partial_t\thet^{n}}{\zeta} +\intO{[-\thet^{n} \vv^{n} \cdot \nabla \zeta + \kappa(\thet^n)\nabla\thet^{n}\cdot \nabla \zeta]} = \intO{\S^{n}: \D\vv^{n} \zeta} \\
&\qquad \mbox{ for all }\zeta\in W_0^{1,2}(\Omega),
\end{aligned}
\end{align}
complemented by the initial conditions $\vv^n(0, \cdot)=\vv^n_0$, $\thet^n(0, \cdot)=\thet_0^n$. 
The functions $\{\vw_j\}_{j\in \N}$ constitute the basis of Hilbert space $W^{3,2}(\Omega; \R^3) \cap L^2_{0, \diver}$ orthonormal in $L^2(\Omega; \R^3)$,
 see \cite[Appendix~6.4]{mnrr1996}, and the function $\vv^n$ is a linear combination of first $n$ elements of the basis and $\S^n$ is just shortcut for composed function, i.e. 
\begin{equation}
  \vv^n:= \sum_{i=1}^n c^{n}_i(t)\vw_i(x) \quad \mbox{ and } \quad   \S^n:= \S^*(\thet^n, \D \vv^n).
  \end{equation}
The function $\vv^n_0$ is the projection of $\vv_0$ onto linear hull of $\{\vw_1,\dots,\vw_n\}$ in $L^2(\Omega,\R^3)$ and $\thet^n_0$ is a suitable mollification of $\thet_0$, for details see \cite[Appendix B, (B.2)-(B.3)]{BuFeMa2009}. They satisfy
\begin{equation}\label{conv:ic}
    \vv^n_0\to\vv_0\quad\mbox{in $L^2(\Omega)$,}\quad \thet_0^n\to\thet_0\quad\mbox{in $L^1(\Omega)$ as $n\to+\infty$,}\quad  \tmin\leq\thet_0^n\quad\mbox{a.e. in $\Omega$.}
\end{equation}
  
The approximations $(\vv^n, \thet^n)$ possess for all $T>0$ the following regularity
\begin{gather}
\label{reg:vn}  \vv^n\in \mathcal{C}^1([0,T];W^{3,2}(\Omega))\\
 \label{eq:approxreg} \thet^n\in L^2(0,T;W^{1,2}(\Omega)),\quad\partial_t\thet^n\in L^2(0,T;(W^{1,2}_0(\Omega))^*),\quad \thet^n\in \mathcal{C}([0,T];L^{2}(\Omega)).
\end{gather}
In addition, the function $\thet^n$ attains the boundary conditions $\thetah$, i.e., $\thet^n - \thetah \in L^2(0,T; W^{1,2}_0(\Omega))$ and fulfills the minimum principle
\begin{equation}\label{eq:minprinn}
\thet^n \ge \underline{\thet}>0 \mbox{ a.e. in }Q.
\end{equation}
We do not present any proof of the existence of $(\vv^n, \thet^n)$. Instead, we refer a reader to the article \cite[Appendix B]{BuFeMa2009}, where a similar problem is treated.

\subsection{Inequalities for approximations}
Here we derive inequalities for approximations of solutions which, after the limit passage $n \to +\infty$, yield the required properties of the constructed weak solution. In Sections~\ref{sec:TE}, \ref{sec:EB} and~\ref{sec:initial} we need to apply the chain rule to the time derivative of $\thet^n$, which is defined only in a dual space. This can be achieved by employing the method of Steklov averages. For details of this argument, we refer the reader to \cite{ABK22,AbBuKa2024}.

\subsubsection{Kinetic energy inequality}
Let us multiply \eqref{ode11} by $c^n_i$, sum over $i=1,\dots,n$, use the fact that $\diver \vv^n=0$, then multiply by $\psi \in \mathcal{C}^{\infty}_0(0,+\infty)$ and integrate over the time interval $(0,+\infty)$. This yields the kinetic energy equality for the Galerkin system
\begin{equation}\label{kinetic-eq-int}
  \intnn{\left(\intO{-|\vv^n|^2 \dert \psi} 
  + 2\intO{\S^n:\D \vv^n \psi}\right)} = 0.
\end{equation}
By the regularity of $\vv^n$ in \eqref{reg:vn}, relation \eqref{kinetic-eq-int} further implies
\begin{equation}\label{kinetic-eq}
  \nm{\vv^n(T)}_2^2 + 2\intTO{\S^n:\D \vv^n} = \nm{\vv^n_0}_2^2,
\end{equation}
for all $T>0$.

\subsubsection{Total energy balance}\label{sec:TE}
Next, let $\varphi \in \mathcal{C}^1(\Omega)$ be nonnegative and satisfy $\varphi = 1$ on $\partial \Omega$. We then set 
$$
  \zeta := 1 - \bb(\thet^n, \thetah)\varphi
$$
in~\eqref{ode22}. Such a choice is legitimate since $\bb(\thet^n,\thetah)=1$ on $\partial \Omega$, and by the regularity of $\bb$ and $\thet^n$ we conclude that for any $t>0$, $\zeta(t) \in W^{1,2}_0(\Omega)$. Substituting this test function into \eqref{ode22}, we obtain
\begin{equation}\label{rur}
\begin{aligned}
&\dual{\partial_t\thet^{n}}{1-\bb(\thet^n, \thetah)\varphi} 
+ \int_{\Omega} -\thet^{n} \vv^{n} \cdot \nabla\!\left[1-\bb(\thet^n, \thetah)\varphi\right] 
+ \kappa(\thet^n)\nabla\thet^{n}\cdot \nabla\!\left[1-\bb(\thet^n, \thetah)\varphi\right] \dx \\
&\qquad = \intO{\S^{n}:\D\vv^{n} \, [1-\bb(\thet^n, \thetah)\varphi]}.
\end{aligned}
\end{equation}
Finally, let $\psi \in \mathcal{C}^\infty_0(0,+\infty)$ be arbitrary. We multiply the above identity by $\psi$ and integrate with respect to $t \in (0,+\infty)$. In what follows, we evaluate all resulting integrals. 

Using the definition \eqref{Bdef}, the fact that $\thetah$ is time-independent, and an integration by parts, the time derivative term becomes
\begin{equation*}
\int_0^{+\infty}\dual{\partial_t\thet^{n}}{[1-\bb(\thet^n, \thetah)\,\varphi]\,\psi}\dt = - \intnO{(\thet^{n} - \B(\thet^n, \thetah)\,\varphi)\, \partial_t \psi}.
\end{equation*}
For the second term in \eqref{rur}, we use the divergence-free condition for $\vv^n$, the zero trace of $\vv^n$, and the integration by parts to conclude
\begin{equation*}
\begin{split}
   \intnO{-\thet^{n} \vv^{n} \cdot \nabla\!\left[1-\bb(\thet^n, \thetah)\varphi\right]  \psi}&=
   \intnO{\thet^{n} \vv^{n} \cdot \nabla\!\left[\bb(\thet^n, \thetah)\varphi\right]  \psi}\\
   &=-\intnO{\vv^{n} \cdot \nabla \thet^{n}  \, \bb(\thet^n, \thetah)\, \varphi\,  \psi}.
\end{split}
\end{equation*}
Finally the elliptic term is evaluated as 
\begin{equation*}
 \begin{array}{l}\displaystyle\vspace{6pt}
\intnO{\kappa(\thet^n)\nabla\thet^{n}\cdot \nabla [1-\bb(\thet^n, \thetah)\,\varphi]\,\psi}\\ \displaystyle\vspace{6pt}
\hfill = - \intnO{\kappa(\thet^n)\nabla\thet^{n}\cdot \nabla \varphi \,\bb(\thet^n, \thetah)\,\psi} - \intnO{\kappa(\thet^n)|\nabla\thet^{n}|^2 \,\partial_1\bb(\thet^n, \thetah)\, \varphi\, \psi}\\ \displaystyle\vspace{6pt}
\hfill-\intnO{\kappa(\thet^n)\nabla\thet^{n}\cdot\nabla\thetah \,\partial_{2} \bb(\thet^n, \thetah)\,\varphi \,\psi}.
\end{array}
\end{equation*}
Collecting all the terms we deduce  the final identity
\begin{equation}\label{temp} 
\begin{aligned}
&- \intnO{(\thet^{n} - \B(\thet^n, \thetah)\varphi) \,\partial_t \psi} -\intnO{\vv^{n}\cdot\nabla \thet^{n}\bb(\thet^n, \thetah)\varphi \psi}\\
&\quad - \intnO{\kappa(\thet^n)\nabla\thet^{n}\cdot \nabla \varphi \,\bb(\thet^n, \thetah)\,\psi} - \intnO{\kappa(\thet^n)|\nabla\thet^{n}|^2 \,\partial_1\bb(\thet^n, \thetah)\, \varphi \,\psi}\\
&\quad -\intnO{\kappa(\thet^n)\nabla\thet^{n}\cdot\nabla\thetah \,\partial_{2} \bb(\thet^n, \thetah)\,\varphi \,\psi}\\
&\qquad  = \intnO{\S^{n}: \D\vv^{n} [1-\bb(\thet^n, \thetah)\,\varphi]\,\psi}.
\end{aligned}
\end{equation}

We sum the last outcome for the temperature balance \eqref{temp} with the kinetic energy balance~\eqref{kinetic-eq-int} and obtain
\begin{equation}\label{2.8}
\begin{array}{l}\displaystyle\vspace{6pt}
  - \intnO{\left[\frac{|\vv^n|^2}{2} + \thet^{n} - \B(\thet^n, \thetah)\,\varphi\right] \,\partial_t \psi}-\intnO{\vv^{n}\cdot\nabla \thet^{n}\bb(\thet^n, \thetah)\varphi \psi}
  \\
 \displaystyle\vspace{6pt}
 - \intnO{\kappa(\thet^n)\nabla\thet^{n}\cdot \nabla \varphi\, \bb(\thet^n, \thetah)\,\psi +\kappa(\thet^n)\nabla\thet^{n}\cdot\nabla\thetah \,\partial_{2} \bb(\thet^n, \thetah)\,\varphi \,\psi}\\ \displaystyle\vspace{6pt}
 + \intnO{\left(\S^{n}: \D\vv^{n} \,\bb(\thet^n, \thetah)-\kappa(\thet^n)|\nabla\thet^{n}|^2 \,\partial_1\bb(\thet^n, \thetah)\right) \,\varphi \,\psi}=0.
\end{array}
\end{equation}


\subsubsection{Entropy balance}\label{sec:EB}
We set $\zeta := \varphi / \thet^n$ as a test function in \eqref{ode22}, where $\varphi \in W^{1,2}_0(\Omega)$ is arbitrary, and directly obtain the following identity for the approximate entropy $\eta^n := \ln \thet^n$:
\begin{equation}
\begin{split}\label{entropy-weak}
&\dual{\dert \eta^n}{\varphi} + \intO{\kappa(\thet^n)\nabla\eta^n \cdot\nabla\varphi - \eta^n\,\vv^n\cdot \nabla\varphi} = \intO{\frac{\S^n:\D{\vv}^n}{\thet^n}\,\varphi+\kappa(\thet^n)|\nabla\eta^n|^2\,\varphi}.
\end{split}
\end{equation}
Compare \cite[Section 3, Part B]{ABK22}. 

\subsection{Limit in the Faedo--Galerkin approximations}In this section, we present the limiting procedure. Many aspects of this procedure are standard; therefore, we focus only on those parts that essentially differ from the existing literature. These parts will be discussed in detail.

\subsubsection{A~priori $n$-independent estimates}
Following, for instance, \cite{BuFeMa2009, BuMaRa09} and introducing $\thetat^n := \thet^n - \thetah$, it is standard to deduce the following $n$-independent a~priori estimates, valid for all $T>0$, $r \in [1,5/3)$, $s \in [1,5/4)$, $a\in (0,1/2)$, and $Q_T = (0,T)\times\Omega$:
\begin{align}
    \label{apr:vn}
      \nm{\vv^{n}}_{L^\infty(0,T; L^2_{0,\diver})}+\nm{\vv^{n}}_{L^p(0,T; W^{1,p}_{0,\diver})}+\nm{\S^n}_{L^{p'}(Q_T)}+\nm{\vv^{n}}_{L^{\frac{5p}3}(Q_T)}&\leq C,
    \\
   \label{timen}
 \nm{\thet^n}_{L^\infty(0,T;L^1(\Omega))}+\nm{\partial_t \thet^n}_{L^1(0, T; (W^{2, 3}_0(\Omega))^*)}+\nm{\partial_t\vv^n}_{L^{\frac{5p}6}(0, T; (W^{1,\frac{5p}{5p-6}}_{0,\diver})^*)} &\leq C,\\
\label{eq:re3}
   \nm{(\thet^n)^a}_{L^2(0,T;W^{1,2}(\Omega))}+\nm{\thet^n}_{L^r(Q_T)}+\nm{\thetat^n}_{L^s(0, T; W_0^{1,s}(\Omega))}&\leq C.
\end{align}
The constant $C>0$ may depend on $a$, $p$, $r$, $s$, and $T$, but it is independent of $n$. In particular, the above estimates may not be uniform as $T \to \infty$. 
We also note that the estimate for the time derivative $\partial_t \thet^{n}$ is not optimal; however, it is sufficient for our purposes, namely for the application of the Aubin--Lions lemma.

By virtue of the established uniform estimates \eqref{apr:vn}--\eqref{eq:re3} and applying the Aubin--Lions compactness lemma, we can extract a subsequence (which we do not relabel) and find $(\vv, \thet, \S)$ such that, for every $T>0$,
\begin{align}
 \vv^n &\rightharpoonup^* \vv && \mbox{ weakly-* in } L^\infty(0, T; L^2_{0,\diver}), \label{Linfty2}\\
 \vv^n &\rightharpoonup \vv &&\mbox{ weakly in } L^p(0, T; W^{1,p}_{0,\diver})\cap W^{1,\frac{5p}6}(0,T; (W^{1,\frac{5p}{5p-6}}_{0,\diver})^*),\label{Lp}\\
 \S^n &\rightharpoonup \S &&\mbox{ weakly in } L^{p'}(\QT; \R^3),\label{Lp-S}\\
 \vv^n&\to\vv &&\mbox{ strongly in }L^q(\QT; \R^3) \mbox{ for any $q\in [1, {5p}/{3})$, and a.e. in $\QT$} \label{v-strong}\\
 (\thet^n)^a&\rightharpoonup   (\thet)^a &&\mbox{ weakly in } L^2(0, T; W^{1,2}(\Omega)) \mbox{ for any } a\in (0, {1}/{2}),\label{thet-weak}\\
 \thet^n &\to \thet &&\mbox{ strongly in $L^r(\QT)$ for any $r\in [1, 5/3)$, and a.e. in $\QT$},\label{thet-strong}\\
  \thet^n -\thetah &\rightharpoonup \thet -\thetah && \mbox{ weakly in } L^s(0, T; W_0^{1, s}(\Omega))  \mbox{ for any } s\in [1, 5/4),\label{grad-thet-weak}\\
\partial_t\thet^n &\rightharpoonup^* \partial_t\thet && \mbox{ weakly$^*$ in } \mathcal{M}(0, T; (W_0^{2, 3}(\Omega))^*).\label{time-thet-weak}                                                   
\end{align}
The above convergence results allow us to pass to the limit $n \to \infty$ in \eqref{ode11} and to deduce \eqref{weak1}. To identify the limiting viscous stress tensor $\S$, we refer to the solenoidal Lipschitz truncation method, which is a standard tool for showing that $\S = \S^*(\thet, \D \vv)$ for our system provided $p > 6/5$; see \cite{DieRuWo2010, BreDieSchwa2013, BeMaRa20, BuGwMaSw12}.

Moreover, using \eqref{thet-strong}, \eqref{eq:minprinn}, and the first part of the uniform estimate \eqref{timen}, it follows from the Fatou lemma that, for all $T>0$,
\begin{equation}
\begin{split}\label{theinf1}
  \thet \geq \tmin \quad \mbox{a.e. in $(0,+\infty)\times\Omega$}, \quad \thet \in L^{\infty}(0,T; L^1(\Omega)).
\end{split}
\end{equation}
Next, we focus more on issues related to the temperature equation.

\subsubsection{Uniform $T$-independent estimates for $\thet$}
In this section, we show that the temperature $\thet$ satisfies \eqref{eq:lnl1}.
For arbitrary $m\in\N$ we define the function $F_m(s)$ as follows
$$
F_m(s):=\left\{\begin{aligned}&0 &&\textrm{for }s\le \tmax,\\
&m(s-\tmax) &&\textrm{for } s\in (\tmax, \tmax + 1/m),\\
&1 &&\textrm{for }s\ge \tmax + 1/m
\end{aligned}
\right.
$$
and its primitive function $P_m(s)$ as 
$$
P_m(s):=\left\{\begin{aligned}&0 &&\textrm{for }s\le \tmax,\\
&\frac{m(s-\tmax)^2}{2} &&\textrm{for } s\in (\tmax, \tmax + 1/m),\\
&s-\tmax-1/(2m) &&\textrm{for }s\ge \tmax + 1/m.
\end{aligned}
\right.
$$
Note that since $\thetah\le \tmax$ and $\thet^n=\thetah$ on $\partial \Omega$, we have that $F_m(\thet^n)\in L^2(0,T;W_0^{1,2}(\Omega))$ and thus we can set $\zeta:= F_m(\thet^n)$ in \eqref{ode22}. Doing so and using the fact that $\diver \vv^n =0$, we obtain the identity 
$$
\partial_t \intO{P_m(\thet^n)}  +\intO{\kappa(\thet^n)F_m'(\thet^n)|\nabla\thet^{n}|^2} = \intO{\S^{n}: \D\vv^{n} F_m(\thet^n)}.
$$
Since $F_m$ is nondecreasing, the second term on the left-hand side is nonnegative and can be neglected. Similarly, the term on the right-hand side can always be estimated by $\S^n : \D\vv^n$ since $F_m \le 1$. Integrating the resulting inequality over time $t \in (0,\tau)$, and using \eqref{kinetic-eq} along with the properties of $P_m(s)$, we obtain for all $\tau \in (0,\infty)$
$$
\begin{aligned}
\intO{P_m(\thet^n(\tau))} &\le \intO{P_m(\thet_0^n)} + \int_0^\tau \intO{\S^n : \D\vv^n} \, d\tau \\
&\le \intO{(\thet_0^n - \tmax)_+} + \intnO{\S^n : \D\vv^n} \le \intO{\frac12 |\vv^n_0|^2 + (\thet_0^n - \tmax)_+}.
\end{aligned}
$$
Letting $m \to \infty$ in the above inequality, we obtain
$$
\intO{(\thet^n(\tau) - \tmax)_+} \le \intO{\frac12 |\vv^n_0|^2 + (\thet_0^n - \tmax)_+},
$$
and consequently, letting $n \to \infty$ and applying the Fatou lemma together with the convergence results \eqref{conv:ic}, we have for almost all $\tau \in (0,\infty)$
$$
\intO{(\thet(\tau) - \tmax)_+} \le \intO{\frac12 |\vv_0|^2 + (\thet_0 - \tmax)_+}.
$$
Thus, using the above estimate, we deduce
$$
\intO{\thet(\tau)} \le \intO{(\thet(\tau) - \tmax)_+ + \tmax} \le \intO{\frac12 |\vv_0|^2 + (\thet_0 - \tmax)_+ + \tmax},
$$
which is precisely the desired estimate \eqref{eq:lnl1}.

\subsubsection{Weak lower semicontinuity of $\S : \D\vv$ and the limit in the kinetic energy inequality}

We claim that the following inequalities hold for any $0 \le s < t < T<\infty$, any nonnegative $\varphi \in L^{\infty}(\Omega)$ and any nonnegative  $\psi \in L^{\infty}(0,T)$:
\begin{align}
\label{eq:wlsc1} 
\intTO{\S : \D\vv \, \bb(\thet, \thetah) \, \varphi \, \psi} 
&\le \liminf_{n \to +\infty} \intTO{\S^{n} : \D\vv^{n} \, \bb(\thet^n, \thetah) \, \varphi \, \psi},\\
\label{eq:wlsc2} 
\intTO{\S : \D\vv \, \varphi\, \psi } 
&\le \liminf_{n \to +\infty} \intTO{\S^{n} : \D\vv^{n}\, \varphi\, \psi},\\
\label{eq:wlsc3} 
\int_s^t \int_\Omega \frac{\S : \D\vv}{\thet} \, \varphi\, \psi\, \dxdt 
&\le \liminf_{n \to +\infty} \int_s^t \int_\Omega \frac{\S^n : \D\vv^n}{\thet^n} \, \varphi\, \psi\, \dxdt.
\end{align}
Since all three inequalities are proved in a similar way, we present only a proof of the first one. We use the monotonicity of $\S^*(\thet^n, \cdot)$ and the nonnegativity of $\bb$, $\varphi$, and $\psi$ to obtain
\begin{equation}
\begin{aligned}\label{eq:221}
  &\int_{Q_T} \S^n : \D\vv^n \, \bb(\thet^n, \thetah) \varphi \psi \, \dxdt
  \ge \int_{Q_T} \S^n : \D\vv \, \bb(\thet^n, \thetah) \varphi \psi \, \dxdt \\
  &\qquad + \int_{Q_T} \S^*(\thet^n, \D\vv) : \D\vv^n \, \bb(\thet^n, \thetah) \varphi \psi \, \dxdt
  - \int_{Q_T} \S^*(\thet^n, \D\vv) : \D\vv \, \bb(\thet^n, \thetah) \varphi \psi \, \dxdt.
\end{aligned}
\end{equation}
Now, we observe that we can pass to the limit $n \to +\infty$ on the right-hand side. Indeed, $\bb(\thet^n, \thetah) \to \bb(\thet, \thetah)$ in $L^q(Q)$ for any $q>1$, because $\bb$ is bounded (see Definition~\ref{def:blafnc}, property \eqref{item0}), and we have the pointwise convergence of $\thet^n$ (see \eqref{thet-strong}). Similarly, using the growth assumption on $\S^*$ stated in \eqref{MBBounds}, the pointwise convergence of $\thet^n$ \eqref{thet-strong}, and the boundedness of $\bb$, we can apply the Lebesgue dominated convergence theorem to obtain
$$
\S^*(\thet^n, \D\vv) \, \bb(\thet^n, \thetah) \varphi \psi 
\to \S(\thet, \D\vv) \, \bb(\thet, \thetah) \varphi \psi \quad \text{strongly in } L^{p'}(Q_T)
$$
as $n \to +\infty$. Using also the convergence results \eqref{Lp} and \eqref{Lp-S}, together with the identification $\S = \S^*(\thet, \D\vv)$, the convergence of the right-hand side of \eqref{eq:221} to 
$$
\int_{Q_T} \S : \D\vv \, \bb(\thet, \thetah) \varphi \psi \, \dxdt
$$ 
is clear. Considering the $\liminf_{n \to +\infty}$ of \eqref{eq:221} then yields \eqref{eq:wlsc1}.
Furthermore, using \eqref{v-strong} together with \eqref{eq:wlsc2} in \eqref{kinetic-eq} directly leads to  \eqref{eq:kinen1}.

\subsubsection{Limit in the $b$ - corrected total energy inequality}\label{sec:lvteb}
Here, we want to let $n \to \infty$ in \eqref{2.8} and deduce \eqref{TE}. The limit passage in the first integral of \eqref{2.8} is straightforward, since the convergence results \eqref{v-strong} and \eqref{thet-strong} hold, and $B$ satisfies the growth assumption~\eqref{rem:B-growth} with negative~$\gamma$.

The second integral in \eqref{2.8} can be rewritten as
$$
\intnO{\vv^{n}\cdot\nabla \thet^{n}\,\bb(\thet^n, \thetah)\varphi \psi}
=\intnO{\vv^{n}\cdot\frac{\nabla[(\thet^{n})^{1+\gamma}]}{1+\gamma}\,
\frac{\bb(\thet^n, \thetah)}{(\thet^n)^\gamma}\,\varphi \psi}.
$$
Since $\gamma+1 \in (0,1/2)$, we can use \eqref{thet-weak} to conclude that  
$\nabla[(\thet^{n})^{1+\gamma}]$ converges weakly to $\nabla[(\thet)^{1+\gamma}]$ in 
$L^2(0,T;L^2(\Omega;\R^3))$ for any $T>0$.  Next, thanks to the assumption \eqref{item0} on $\bb$ and the strict positivity of $\thet^n$ (see \eqref{eq:minprinn}), we infer that $\bb(\thet^n,\thetah)(\thet^n)^{-\gamma}$ is bounded. Thus, using \eqref{v-strong} and \eqref{thet-strong}, we obtain that 
$$
\vv^{n}\,\bb(\thet^n,\thetah)(\thet^n)^{-\gamma} \to 
\vv\,\bb(\thet,\thetah)(\thet)^{-\gamma} 
\quad\text{strongly in } L^2(0,T;L^2(\Omega;\R^3)).
$$
Combining these two convergence results, and since $\psi$ is compactly supported in $(0,\infty)$, we deduce that  
$$
\begin{aligned}
&\lim_{n\to \infty}\intnO{\vv^{n}\cdot\nabla \thet^{n}\,\bb(\thet^n, \thetah)\varphi \psi}
=\lim_{n\to \infty}\intnO{\vv^{n}\cdot\frac{\nabla[(\thet^{n})^{1+\gamma}]}{1+\gamma}\,
\frac{\bb(\thet^n, \thetah)}{(\thet^n)^\gamma}\,\varphi \psi} \\
&\qquad =\intnO{\vv\cdot\frac{\nabla[(\thet)^{1+\gamma}]}{1+\gamma}\,
\frac{\bb(\thet, \thetah)}{(\thet)^\gamma}\,\varphi \psi} =\intnO{\vv\cdot\nabla \thet\,\bb(\thet, \thetah)\varphi \psi}.
\end{aligned}
$$

Using a similar procedure, we rewrite the third integral in \eqref{2.8} as follows:
$$
\begin{aligned}
&\intnO{\kappa(\thet^n)\nabla\thet^{n}\cdot \nabla \varphi\, \bb(\thet^n, \thetah)\,\psi 
+ \kappa(\thet^n)\nabla\thet^{n}\cdot\nabla\thetah \,\partial_{2} \bb(\thet^n, \thetah)\,\varphi \,\psi} \\
&\qquad = \frac{1}{1+\gamma}\intnO{\psi \,\kappa(\thet^n)\nabla (\thet^{n})^{1+\gamma} \cdot 
\left(\nabla \varphi\, \frac{\bb(\thet^n, \thetah)}{(\thet^n)^{\gamma}} 
+ \nabla\thetah \, \frac{\partial_{2}\bb(\thet^n, \thetah)}{(\thet^n)^{\gamma}}\,\varphi \right)}.
\end{aligned}
$$
Since $\bb$ and $\kappa$ are bounded, and $\bb$ satisfies the growth assumptions \eqref{item0} and \eqref{item2}, we can use arguments analogous to those applied for the preceding integral and directly identify the limit. 

We will not pass to the limit in the last term of \eqref{2.8}; instead, we show that it is nonnegative. Indeed, it follows directly from the nonnegativity of $-\partial_1 b$ (see \eqref{item3}), the assumptions \eqref{MBBounds}, the nonnegativity of $b$ (see \eqref{item0}), and the nonnegativity of $\varphi$ and $\psi$ that
$$
\intnO{\left(\S^{n}: \D\vv^{n} \,\bb(\thet^n, \thetah) 
- \kappa(\thet^n)|\nabla\thet^{n}|^2 \,\partial_1\bb(\thet^n, \thetah)\right) \varphi \psi} \ge 0.
$$
Combining all above convergence results, we can let $n\to \infty$ in \eqref{2.8} to deduce \eqref{TE}.

\subsubsection{Limit in the entropy balance}\label{sec:leb} Let $\varphi \in \mathcal{C}^1_0(\Omega)$ and $\psi\in \mathcal{C}_0^1(0,T)$ with some $T>0$ be arbitrary and nonnegative. We multiply \eqref{entropy-weak} by $\psi$ and integrate with respect to $t\in (0,T)$. Using integration by parts, we obtain 
\begin{equation}
\begin{split}\label{entropy-weak2}
&\intTO{ -\eta^n \varphi \dert\psi  + \kappa(\thet^n)\nabla\eta^n \cdot\nabla\varphi \psi - \eta^n\,\vv^n\cdot \nabla\varphi \psi} \\
&\qquad = \intTO{\frac{\S^n:\D{\vv}^n}{\thet^n}\,\varphi \psi +\kappa(\thet^n)|\nabla\eta^n|^2\,\varphi \psi}.
\end{split}
\end{equation}
By virtue of the established convergence results \eqref{thet-weak} and \eqref{thet-strong}, 
we deduce that $\eta^n \to \eta$ in $L^q(Q_T)$ for any $q>1$ and 
$\nabla\eta^n \rightharpoonup \nabla\eta$ in $L^2(Q_T;\R^3)$. 
These convergence results allow us to pass to the limit as $n\to+\infty$ in all terms on the left-hand side of \eqref{entropy-weak2} and thus obtain the left-hand side of \eqref{entropy-limit}. 

Next, we focus on the right-hand side of \eqref{entropy-weak2}. 
It is sufficient to establish convergence with the inequality sign in order to deduce \eqref{entropy-limit}. 
The first term on the right-hand side has already been treated in \eqref{eq:wlsc3}. 
A similar result holds for the second term: it is enough to use the weak lower semicontinuity of convex functionals together with the above-mentioned convergence results for $\eta^n$ and $\nabla \eta^n$. 
Indeed, we have
$$
\begin{aligned}
\liminf_{n\to+\infty}\int_{Q_T}\kappa(\thet^n)|\nabla\eta^n|^2\varphi\, \psi \dxdt &\geq\lim_{n\to+\infty}\int_{Q_T}\kappa(\thet^n)(2\nabla\eta^n\nabla\eta-|\nabla\eta|^2)\varphi\, \psi \dxdt \\
&=\int_{Q_T}\kappa(\thet)|\nabla\eta|^2\varphi\, \psi \dxdt.
\end{aligned}
$$
Combining all results in this paragraph gives \eqref{entropy-limit}.

\subsection{Attaining of the initial condition}\label{sec:initial}
Using the identity \eqref{kinetic-eq}, we can also follow the classical procedure and deduce that 
$\lim_{t\to0}\|\vv(t)-\vv_0\|_{2}=0$, see e.g.~\cite[Corollary~4.7]{RoRoSa2016}. 
Next, we establish an analogous result for the temperature $\thet$. 

First, we show that $\sqrt{\vartheta} \in \mathcal{D}_{\textrm{weak}}(0,T; L^2(\Omega))$ for all $T>0$. Fix $T>0$. Thanks to \eqref{eq:lnl1}, we already know that $\sqrt{\vartheta}\in L^{\infty}(0,T; L^2(\Omega))$. 
Therefore, to prove the claim, it is sufficient to establish that 
$\partial_t \sqrt{\vartheta} \in \mathcal{M}(0,T; (W^{2,3}_0(\Omega))^*)$. 

To this end, we derive the equation for $\sqrt{\vartheta^n}$. 
We set $\zeta:=\frac{\varphi}{2\sqrt{\vartheta^n}}$ in \eqref{ode22}, 
where \pk{$\varphi\in C^\infty_0(\Omega)$} is arbitrary, and obtain
\begin{equation}\label{ode222}
\dual{\partial_t\sqrt{\vartheta^{n}}}{\varphi} 
+ \intO{-\sqrt{\vartheta^{n}}\, \vv^{n} \cdot \nabla \varphi 
+ \kappa(\vartheta^n)\nabla \sqrt{\vartheta^{n}}\cdot \nabla \varphi} 
= \intO{\frac{\S^{n}: \D\vv^{n}\varphi}{2\sqrt{\vartheta^n}} 
+\frac{\kappa(\vartheta^n)|\nabla\vartheta^{n}|^2\varphi}{4 (\vartheta^n)^{\frac32}}}.
\end{equation}

It follows from the uniform estimates \eqref{apr:vn}--\eqref{eq:re3} and \eqref{theinf1} that
$$
\int_0^T \big\|\partial_t \sqrt{\vartheta^n}\big\|_{(W^{2,3}_0(\Omega))^*}\,\mathrm{d}t \le C.
$$
Consequently (for a subsequence),
$$
\partial_t \sqrt{\vartheta^n} \rightharpoonup^* \partial_t \sqrt{\vartheta} 
\quad\text{weakly$^*$ in } \mathcal{M}(0,T;  (W^{2,3}_0(\Omega))^*).
$$
Hence, combined with \eqref{eq:lnl1}, we conclude that
\begin{equation}\label{cadlag}
\sqrt{\vartheta} \in \mathcal{D}_{\textrm{weak}}(0,T; L^2(\Omega)).
\end{equation}

Similarly, one reads from \eqref{entropy-weak} that for a subsequence
$$
\partial_t {\eta^n} \rightharpoonup^* \partial_t \eta 
\quad\text{weakly$^*$ in } \mathcal{M}(0,T;  (W^{2,3}_0(\Omega))^*).
$$

Next, assuming that $\varphi \geq 0$, we integrate \eqref{ode222} with respect to 
$t \in (0,\tau)$ with some $\tau>0$. Using the nonnegativity of the right-hand side. We obtain
$$
\intO{\sqrt{\vartheta^n(\tau)}\,\varphi}
- \inttO{\sqrt{\vartheta^n}\,\vv^n \cdot \nabla \varphi} 
+ \inttO{\kappa(\vartheta^n)\nabla\sqrt{\vartheta^n}\cdot \nabla \varphi}
\;\geq\;\intO{\sqrt{\vartheta^n_0}\,\varphi}.
$$
By virtue of \eqref{v-strong}--\eqref{thet-strong} and \eqref{conv:ic}, 
we can pass to the limit $n\to\infty$ in the above inequality and deduce that, 
for almost all $\tau \in (0,+\infty)$,
\begin{equation}\label{new1}
\intO{\sqrt{\vartheta(\tau)}\,\varphi}
- \inttO{\sqrt{\vartheta}\,\vv \cdot \nabla \varphi} 
+ \inttO{\kappa(\vartheta)\nabla\sqrt{\vartheta}\cdot \nabla \varphi}
\;\geq\;\intO{\sqrt{\vartheta_0}\,\varphi}.
\end{equation}
Thanks to \eqref{cadlag}, we know that 
\begin{equation}\label{weakonly}
\sqrt{\vartheta(\tau)} \rightharpoonup \sqrt{\vartheta(0)} 
\quad \text{weakly in } L^2(\Omega) \text{ as } \tau \to 0_+.
\end{equation}
Therefore, letting $\tau \to 0_+$ in \eqref{new1}, we deduce
\begin{equation}\label{eq:thet-liminf}
 \sqrt{\vartheta(0)} \geq \sqrt{\vartheta_0}
 \qquad \text{a.e. in } \Omega.
\end{equation}


Our goal now is to show that \eqref{eq:thet-liminf} holds with the equality sign and that the convergence result \eqref{weakonly} holds even in the strong topology. First, we show this claim only locally in $\Omega$. To do so, let $\varphi \in \mathcal{C}^1_0(\Omega)$ be arbitrary with $0 \le \varphi \le 1$. 
We set $\zeta := \varphi^2$ in \eqref{ode22} and integrate the result over $t \in (0,\tau)$ with some $\tau>0$. We add this result to \eqref{kinetic-eq}, which we divided by two and in which we set $T := \tau$, and obtain
\begin{equation*}\begin{split}
\intO{\thet^n(\tau)\varphi^2+\tfrac12|\vv^n(\tau)|^2}
&+\inttO{-\thet^n\vv^n\cdot \nabla\varphi^2+\kappa(\thet^n)\nabla\thet^n \cdot \nabla\varphi^2} \\
&\leq\intO{\thet_0^n\varphi^2+\tfrac12|\vv_0^n|^2}.
\end{split}\end{equation*}
Note that the dissipative heating term in the heat equation was cancelled with the corresponding term in the kinetic energy equality, and the positive remainder was dropped. This was possible due to the assumption $0 \le \varphi \le 1$ in $\Omega$.

As above, we can pass to the limit as $n \to +\infty$ for almost all $\tau \in (0,+\infty)$ to deduce 
\begin{equation}\label{local-st}
\begin{split}
\intO{\thet(\tau)\varphi^2+\tfrac12|\vv(\tau)|^2}
&+\inttO{-\thet\vv \cdot\nabla\varphi^2+\kappa(\thet)\nabla\thet \cdot \nabla\varphi^2}\\
&\leq\intO{\thet_0\varphi^2+\tfrac12|\vv_0|^2}.
\end{split}
\end{equation}
Next, we strenghten the above result so that it holds for all $\tau \in (0,+\infty)$. Intergrating the above inequality with respect to $\tau \in (\sigma, \sigma+\delta)$ with $\sigma>0$, dividing by $\delta$ and letting $\delta \to 0_+$, we see
\begin{equation*}\begin{split}
\liminf_{\delta \to 0} \frac{1}{\delta}\int_\sigma^{\sigma+\delta}\intO{|\sqrt{\thet(\tau)}\varphi|^2+\tfrac12|\vv(\tau)|^2}\; {\rm d}\, \tau
&+\int_0^\sigma\intO{-\thet\vv \cdot\nabla\varphi^2+\kappa(\thet)\nabla\thet \cdot \nabla\varphi^2}\dd t\\
&\leq\intO{\thet_0\varphi^2+\tfrac12|\vv_0|^2}
\end{split}
\end{equation*}
Since $\vv \in \mathcal{C}_{\textrm{weak}}(0,T; L^2_{0,\diver})$ and $\sqrt{\thet} \in \mathcal{D}_{\textrm{weak}}(0,T; L^2(\Omega))$ for any $T>0$, we have, as $\tau\to \sigma_+$,  
$$
\begin{aligned}
\vv(\tau) &\rightharpoonup \vv(\sigma) &&\textrm{weakly in } L^2_{0,\diver},\\
\sqrt{\thet(\tau)}\varphi &\rightharpoonup \sqrt{\thet(\sigma)}\varphi &&\textrm{weakly in } L^2(\Omega).
\end{aligned}
$$
Therefore, by a simple algebraic manipulation, we obtain
$$
\begin{aligned}
&\liminf_{\delta \to 0} \frac{1}{\delta}\int_\sigma^{\sigma+\delta}\intO{|\sqrt{\thet(\tau)}\varphi|^2+\tfrac12|\vv(\tau)|^2}\, {\rm d}\tau \\[0.3em]
&\quad \ge \liminf_{\delta \to 0} \frac{1}{\delta}\int_\sigma^{\sigma+\delta}\intO{-(|\sqrt{\thet(\sigma)}\varphi|^2+\tfrac12|\vv(\sigma)|^2)}\, {\rm d}\tau \\[0.3em]
&\qquad+\liminf_{\delta \to 0} \frac{2}{\delta}\int_\sigma^{\sigma+\delta}\intO{\sqrt{\thet(\sigma)}\varphi\sqrt{\thet(\tau)}\varphi+\tfrac12\vv(\sigma)\cdot \vv(\tau)}\, {\rm d}\tau\\[0.3em]
&\quad = \intO{|\sqrt{\thet(\sigma)}\varphi|^2+\tfrac12|\vv(\sigma)|^2} \\[0.3em]
&\qquad+\liminf_{\delta \to 0} \frac{2}{\delta}\int_\sigma^{\sigma+\delta}\intO{\sqrt{\thet(\sigma)}\varphi(\sqrt{\thet(\tau)}\varphi-\sqrt{\thet(\sigma)}\varphi) +\tfrac12\vv(\sigma)\cdot (\vv(\tau)-\vv(\sigma))}\, {\rm d}\tau\\[0.3em]
&\quad = \intO{|\sqrt{\thet(\sigma)}\varphi|^2+\tfrac12|\vv(\sigma)|^2}.
\end{aligned}
$$
Consequently, the inequality \eqref{local-st} holds for all $\tau \in (0,+\infty)$.  
Since we already know that $\lim_{\tau\to 0_+}\nm{\vv(\tau)-\vv_0}_2=0$, we can pass to the $\limsup$ and obtain
\begin{equation}\label{eq:thet-limsup}
  \limsup_{\tau\to 0_+}\intO{|\sqrt{\thet(\tau)}\varphi|^2}\leq\intO{|\sqrt{\thet_0}\varphi|^2}.
\end{equation}
Combining \eqref{eq:thet-liminf} with \eqref{eq:thet-limsup}, and recalling that $0\le \varphi\le 1$ was arbitrary, we conclude
$$
\begin{aligned}
\limsup_{\tau \to 0_+} \intO{|(\sqrt{\thet(\tau)} - \sqrt{\thet_0})\varphi |^2}
&=\lim_{\tau \to 0_+} \intO{\thet(\tau)\varphi^2 + \thet_0 \varphi^2 - 2\sqrt{\thet(\tau)}\varphi\sqrt{\thet_0}\varphi} \\[0.3em]
&\le 2\intO{\thet_0 \varphi^2 - \sqrt{\thet(0)}\varphi\sqrt{\thet_0}\varphi}\le 0,
\end{aligned}
$$

which directly implies that, for any compact $K \subset \Omega$,  
\begin{align}
&\sqrt{\thet(\tau)}\to\sqrt{\thet_0}\quad\mbox{ strongly in $L^2(K)$, \label{loc-s}}\\
&\sqrt{\thet(\tau)}\rightharpoonup\sqrt{\thet_0}\quad\mbox{ weakly in $L^2(\Omega)$}, \label{eq:thet-weak0}
\end{align}
as $\tau\to 0_+$.

Finally, we want to strengthen \eqref{loc-s} so that the strong convergence holds on the whole domain $\Omega$. First note that, thanks to \eqref{loc-s}--\eqref{eq:thet-weak0}, we have 
\begin{align}
&\sqrt{\thet(\tau)}\to\sqrt{\thet_0}\quad \textrm{ strongly in } L^q(\Omega)  \textrm{ for all } q\in [1,2). \label{glob-s}
\end{align}
For the next computation we recall the definition~\eqref{DFG}, which for $s,t>0$ reads
\begin{equation}
\mathcal{H}^\al(s):= \int_1^s (\G(z))^{-\al}\,{\rm d}z,\quad 
\G(s):=\int_0^s \kappa(z)\, {\rm d} z,\quad 
\ff_\alpha(s,t):=s-t - (\mathcal{H}^\al(s) - \mathcal{H}^\al(t))(\G(t))^{\al}.
\end{equation}
The role of the function $\G$ is to simplify the situation in the case of nonconstant heat conductivity. 
It will be even more important in the next section. 
Note that if $\kappa$ is constant, then the main part of $\mathcal{H}^\al(s)$ is just a multiple of the power function $s^{1-\alpha}$.

We set $\zeta:= 1-\bigl({\G(\thetah)/\G(\thet^n)}\bigr)^\alpha \chi_{(0,\tau)}$ in \eqref{ode22} for a fixed $\alpha>1/2$, 
integrate the result over $t\in (0,\tau)$, and add it to \eqref{kinetic-eq} divided by two (similarly as above). 
This leads to 
\begin{multline} 
    \intO{\tfrac12|\vv^n(\tau)|^2+\ff_\alpha(\thet^n(\tau),\thetah)}\\
    \leq \inttO{\nabla\thet^n \cdot \vv^n \left[\frac{\G(\thetah)}{\G(\thet^n)}\right]^\al}
      +\intO{\tfrac12|\vv_0^n|^2+\ff_\alpha(\thet_0^n,\thetah)}.
\end{multline}
We do not provide detailed computation here, since a similar one is carried out in Section~\ref{SectEntrop} in a more complicated situation. 
We only mention that some terms were dropped due to their positivity, and we also used the equation for $\thetah$ in~\eqref{MB6}.

Now, we pass to the limit as $n\to+\infty$. 
For the first term on the right-hand side we need the convergence
\[
    \frac{\nabla\thet^n}{\G(\thet^n)^\al}\rightharpoonup\frac{\nabla\thet}{\G(\thet)^\al} 
    \quad\mbox{in $L^2((0,T)\times\Omega)$,}
\]
which follows for $\alpha>1/2$, at least for a subsequence, from \eqref{thet-weak} and \eqref{thet-strong}. 
Then we proceed as before to obtain after the limit passage $n\to+\infty$ that for almost all $\tau \in (0,T)$
\begin{equation}\label{stop}
\begin{split} 
&\intO{\ff_\alpha(\thet(\tau),\thetah)-\ff_\alpha(\thet_0,\thetah)}\\
    &\qquad \leq \inttO{\nabla\thet \cdot \vv \left[\frac{\G(\thetah)}{\G(\thet)}\right]^\al}
      +\intO{\tfrac12|\vv_0|^2-\tfrac12|\vv(\tau)|^2}.
\end{split}
\end{equation}


Now we need to confirm that $\ff_\al$ suitably measures a distance and that it can be properly estimated from below. 
We claim that there exists $\alpha > \frac12$ such that for all $0 < s, t, r$ there holds
\begin{equation}\label{zaba}
\ff_\alpha(s^2,t)-\ff_\alpha(r^2,t)-2r\partial_1\ff_\alpha(r^2,t)(s-r) -(s-r)^2 \ge 0.
\end{equation}
Using this inequality in \eqref{stop} with the setting $s:=\sqrt{\thet(\tau)}$, $t:=\thetah$ and $r:=\sqrt{\thet_0}$, 
we obtain for almost all $\tau\in (0,+\infty)$
\begin{equation}\label{stop2}
\begin{split} 
&\intO{|\sqrt{\thet(\tau)} - \sqrt{\thet_0}|^2 }\\
    &\quad \leq \inttO{\nabla\thet \cdot \vv^n \left[\frac{\G(\thetah)}{\G(\thet)}\right]^\al}
      +\intO{\tfrac12|\vv_0|^2-\tfrac12|\vv(\tau)|^2} \\[0.5em]
      &\qquad - \intO{2 \sqrt{\thet_0}\left[1-\left (\frac{\G(\thetah)}{\G(\thet_0)} \right)^\alpha \right](\sqrt{\thet(\tau)}-\sqrt{\thet_0})}.
\end{split}
\end{equation}
Moreover, using convexity and the fact that $\sqrt{\thet}\in \mathcal{D}_{\textrm{weak}}(0,T; L^2(\Omega))$, 
the above inequality holds for every $\tau \in (0,+\infty)$. 
Finally, using the strong convergence $\vv(\tau) \to \vv_0$ and the weak convergence result \eqref{eq:thet-weak0}, 
we can let $\tau \to 0_+$ in the above inequality to conclude 
(note that $\frac{\G(\thetah)}{\G(\thet_0)}$ is bounded)
\begin{equation}\label{stop3}
\lim_{\tau \to 0_+}\intO{|\sqrt{\thet(\tau)} - \sqrt{\thet_0}|^2 }\le 0.
\end{equation}
The fact \eqref{eq:att-ic} then follows since $\sqrt\thet\in L^\infty(0,+\infty;L^2(\Omega))$ and 
\begin{multline*}
  \intO{|\thet(\tau)-\thet_0|}\leq\intO{|\sqrt{\thet(\tau)}-\sqrt{\thet_0}||\sqrt{\thet(\tau)}+\sqrt{\thet_0}|} \\
  \leq C(\nm{\sqrt\thet}_{L^\infty(0,+\infty;L^2(\Omega))}^2+\nm{\sqrt\thet_0}_{L^2(\Omega)}^2)\intO{|\sqrt{\thet(\tau)}-\sqrt{\thet_0}|^2}\to 0,
\end{multline*}
as $\tau \to 0_+$.

It remains to prove \eqref{zaba}. For fixed $t,r >0$ let us  define
$$
A(s):=\ff_\alpha(s^2,t)-\ff_\alpha(r^2,t)-2r\partial_1\ff_\alpha(r^2,t)(s-r) -(s-r)^2
$$
and evaluate its first two derivatives
$$
\begin{aligned}
A'(s)&=2s\partial_1\ff_\alpha(s^2,t)-2r\partial_1\ff_\alpha(r^2,t) -2(s-r),\\
A''(s)&=2\partial_1\ff_\alpha(s^2,t)+4s^2\partial^2_{1}\ff_\alpha(s^2,t)-2.
\end{aligned}
$$
It is clear that $A(r)=A'(r)=0$ and therefore to show the claim, we just need to show that $A''(s)\ge 0$ for all $s\ge \tmin$. Using the definition of $\ff_\alpha(s^2,t)$ we have
$$
\begin{aligned}
A''(s)&= -\frac{2 (\G(t))^{\al}}{(\G(s^2))^{\alpha}} + \frac{4s^2\alpha (\G(t))^{\al} \kappa(s^2)}{(\G(s^2))^{\alpha+1}}\\
 &=\frac{2 (\G(t))^{\al}}{(\G(s^2))^{\alpha}} \left(\frac{2s^2\alpha \kappa(s^2)}{(\G(s^2))}  -1 \right)\ge \frac{2 (\G(t))^{\al}}{(\G(s^2))^{\alpha}} \left(\frac{2\alpha \underline{\kappa}}{\overline{\kappa}}  -1 \right).
\end{aligned}
$$
Hence, setting e.g. $\alpha:=\frac{\overline{\kappa}}{2\underline{\kappa}}>\frac12$, we get the claim. The proof is complete.

%
%

\section{Stability of steady state}
  \label{sec:conv}

This section is devoted to the proof of Theorem~\ref{TMB}. 
We largely follow the approach of \cite{AbBuKa2024}, omitting details where the arguments are analogous. 
However, since we consider a completely new type of solution, namely \bws, 
and cannot rely on an energy equality, a more careful analysis is required. 
Accordingly, all genuinely new steps are presented in full detail, while standard arguments are treated more briefly.

Thus, in this section we assume that we are given a four-tuple $(\vv,\S,\thet,\eta)$ which constitutes a \bws\ in the sense of Definition~\ref{proper-sol}, corresponding to some $(-\alpha)$-\blafnc\ that satisfies all the requirements of Definition~\ref{def:blafnc} with $\alpha \in (0,1)$.

\subsection{Long time behaviour of the velocity field}\label{sec:vlt}
Next we show that the exponential decay of the velocity field follows from \eqref{eq:kinen1}. First we realize that from \eqref{eq:kinen1} and \eqref{eq:att-ic} we get that   for almost all $s,t\in(0,+\infty)$ fulfilling $s<t$ 
  \begin{equation}\label{eq:kinen}
\intO{|\vv(t)|^2} + 2\int_s^t\intO{\S:\D\vv} \leq \intO{|\vv(s)|^2}.
\end{equation}
It directly follows from the above inequality that
\begin{equation}\label{3.1.5}
  \vv \in L^{\infty}(0,+\infty; L^2_{0,\diver}),\quad \S:\D\vv\in L^1((0,+\infty)\times\Omega).
\end{equation}

If $p<2$, we estimate the second term on the left-hand side as
\begin{equation}\begin{split}
\intO{|\D\vv|^p} &= \intO{\left((\delta+|\D\vv|^2)^{\frac{p-2}{2}}|\D\vv|^2\right)^{\frac{p}{2}} (\delta+|\D\vv|^2)^{\frac{2-p}{2} \frac{p}{2}}} \\
&\leq C\left(\intO{\S:\D\vv}\right)^{\frac{p}{2}} \left(\intO{(\delta+|\D\vv|^2)^{\frac{p}{2}}}\right)^{\frac{2-p}{2}}.
\end{split}\end{equation}
Using also the Sobolev embedding $W^{1,p}\hookrightarrow L^2$, we obtain
\begin{equation}\label{eq:sdv}\begin{split}
C\min\left\{ \|\vv\|_{L^2(\Omega)}^2, \|\vv\|_{L^2(\Omega)}^p\right\}
&\leq C\min\left\{ \|\D\vv\|_{L^p(\Omega)}^2, \|\D\vv\|_{L^p(\Omega)}^p\right\} \\
&\leq \frac{\|\D\vv\|_{L^p(\Omega)}^2}{C + \|\D\vv\|_{L^p(\Omega)}^{2-p}}
\leq \intO{\S:\D\vv}.
\end{split}\end{equation}
Application of this inequality in \eqref{eq:kinen} yields the existence of $\mu>0$ such that for almost all $s,t\in(0,+\infty)$ with $s<t$,
\begin{equation}\label{in:bad}
\intO{|\vv(t)|^2} + \mu\int_s^t\min\left\{ \|\vv\|_{L^2(\Omega)}^2(\tau), \|\vv\|_{L^2(\Omega)}^p(\tau)\right\}\, \mathrm{d}\tau
\leq \intO{|\vv(s)|^2}.
\end{equation}

If $p\geq 2$, one can proceed similarly to get the existence of $\mu>0$ such that for almost all $s,t\in(0,+\infty)$ with $s<t$,
\begin{equation}\label{in:good}
\intO{|\vv(t)|^2} + \mu\int_s^t\|\vv\|_{L^2(\Omega)}^2(\tau) \mathrm{d}\tau
\leq \intO{|\vv(s)|^2}.
\end{equation}

We introduce the shorthand notation $f(\tau)=\|\vv(\tau)\|_{L^2(\Omega)}^2$. From \eqref{eq:kinen} it follows that $f(t)\leq f(s)$ for a.e. $s,t\in(0,+\infty)$ with $s<t$. Moreover, since $\vv \in \mathcal{C}_{\textrm{weak}}(0,T; L^2_{0,\diver})$ for all $T>0$, the function $f$ is lower semicontinuous.

We distinguish two cases: 
\begin{itemize}
\item[a)] $p<2$ and $\|\vv_0\|_2>1$, i.e., we have to deal with the inequality \eqref{in:bad},
\item[b)] the remaining cases, i.e., the inequality \eqref{in:good} holds.
\end{itemize}

In case~b), \eqref{in:bad} or \eqref{in:good} for $p\geq 2$ or $p<2$ respectively  implies for a.e. $s, t\in(0,+\infty)$ with $s<t$,
$$
f(t) + \mu\int_s^t f(\tau)\,\mathrm{d}\tau \leq f(s).
$$
An application of Lemma~\ref{lem:decay} then implies that for almost all $s,t\in(0,+\infty)$, $s<t$
\begin{equation}\label{est:st}
f(t)\le e^{-\mu(t-s)}f(s).
\end{equation}
Thanks to \eqref{eq:att-ic}, we may let $s\to 0$ and conclude that for almost all $t\in (0,\infty)$
$$
f(t)\le e^{-\mu t}\|\vv_0\|_2^2.
$$
Finally, since $f$ is lower semicontinuous, the above inequality extends to all $t\in (0,+\infty)$, which is precisely the desired result~\eqref{eq:twocrosses}.

In case~a) we define $\tstar:=\sup\{\tau;\ f>1 \ \text{a.e. in } (0,\tau)\}>0$. Positivity of $t^*$ follows from \eqref{eq:att-ic}. We observe from \eqref{in:bad} that for a.e. $s,t\in(0,t^*)$ with $s<t$
$$
\max(1,f(t) + \mu\int_s^t f(\tau)^{\frac p2}\,\mathrm{d}\tau) \leq f(s).
$$
Lemma~\ref{lem:decay} together with \eqref{eq:att-ic} gives $f(t)\leq \left(\frac{f(0)}{1+\mu t}\right)^{2/p}$ for a.e. $t\in(0,\tstar)$, which implies that $\tstar\leq (f(0)-1)/\mu$.
It follows from the definition of $\tstar$ and properties of $f$ that $f\le 1$ a.e. in $(\tstar,+\infty)$ so we can apply the previous step to deduce for a.e. $s,t\in(\tstar,+\infty)$ with $s<t$ that \eqref{est:st} holds.
Combining estimates on $(0,\tstar)$ and on $(\tstar,+\infty)$ with properties of $f$ we infer that for a.e. $t>0$,
$$
f(t)
\leq e^{-\mu t}\, e^{f(0)-1}\, f(0)^{\frac2p}.
$$
By lower semicontinuity of $f$ we get that the estimate holds actually for all $t>0$ which is exactly~\eqref{snad}.

\subsection{Further properties of \texorpdfstring{\bws}{b}}
In this part we re-establish the a~priori estimates for $\thet$ and $\eta$. 
This is necessary, since we aim to prove the result for an arbitrary \bws, whereas the definition of \bws\ imposes only limited regularity on $\thet$. 
On the other hand, the existence of a solution has already been proved, and the corresponding estimates were derived along the way. 
Here we carefully justify that the same bounds hold for any such \bws.

\subsubsection{Steklov average of the entropy inequality}
Since all estimates for $\thet$ are deduced from \eqref{entropy-limit}, we need to introduce its regularized version, which allows us to deal with the time derivative of the solutions. For any function $u:(0,\infty)\times \Omega \to \mathbb{R}$ and any $t, h>0$, $x\in\Omega$ we denote
$$
 u_h(t, x) = \fint_t^{t+h} u(\tau, x) \,{\rm d}\tau.
$$
Next, let $T>0$, $\varphi \in \mathcal{C}^1_0(\Omega)$ and $\psi \in \mathcal{C}_0^{\infty}((0,T))$ be arbitrary nonnegative functions.  
If we replace $\psi$ by $-\psi_{-h}$ in \eqref{entropy-limit} with sufficiently small $h$, 
then after applying Fubini’s theorem we deduce that
\begin{equation*}
\begin{split}
&\intTO{\dert\eta_h\,\psi \varphi} + \intTO{(\nabla\eta\cdot \vv)_h \, \varphi \,\psi}  + \intTO{(\kappa(\thet)\nabla\eta)_h \cdot\nabla \varphi \, \psi}
\\
& \geq \intTO{\left(\frac{1}{\thet}\,\S:\D{\vv}\right)_h\, \varphi \, \psi} + \intTO{\left(\kappa(\thet)|\nabla\eta|^2\right)_h\,\varphi\, \psi}.
\end{split}
\end{equation*}
Using density and weak$^*$ density arguments, we conclude that the above inequality can be generalized.  
Thus, for all nonnegative $\zeta \in L^\infty((0,T)\times\Omega)\cap L^2(0,T;W^{1,2}_0(\Omega))$ it holds that
\begin{equation}\label{entropy-limit2}
\begin{split}
&\intTO{\dert\eta_h\,\zeta} + \intTO{(\nabla\eta\cdot \vv)_h \, \zeta}  + \intTO{(\kappa(\thet)\nabla\eta)_h \cdot\nabla \zeta}
\\
& \geq \intTO{\left(\frac{1}{\thet}\,\S:\D{\vv}\right)_h\, \zeta} + \intTO{\left(\kappa(\thet)|\nabla\eta|^2\right)_h\,\zeta},
\end{split}
\end{equation}
which is the starting point for further estimates.

\subsubsection{Regularity of $\thet$ and $\vv$}\label{sss:regthetv} Since we are working with \bws, we have \eqref{MBap} at our disposal.  
Nevertheless, we establish more delicate entropy and temperature estimates in what follows.

For any $k\in \N$ and $s\in\R$, we define the standard cut-off function
$$
\mathcal{T}_k(s):={\rm sign}(s) \min\{|s|,k\}.
$$
Furthermore, we define its mollification as
$$
\mathcal{T}_{k,\varepsilon}(s):= \left\{ \begin{aligned}
&s &&\textrm{if }|s|\le k-\varepsilon,\\
&{\rm sign}(s) k &&\textrm{if } |s|\ge k+\varepsilon,\\
&{\rm sign}(s) \left( k - \frac{(k+\varepsilon - |s|)^2}{4\varepsilon} \right) &&\textrm{if }  k-\varepsilon < |s| < k+\varepsilon.
\end{aligned}
\right.
$$

Let $T>0$ be a Lebesgue point of $\eta\in L^\infty(0,+\infty; L^1(\Omega))$, and let $\psi=\psi(t)$ be a nonnegative function in $C^{0,1}([0,T])$ with $\psi(0)=\psi(T)=0$.
We choose a sequence $\{h_n\}$ such that $h_n>0$, $h_n\to 0$ as $n\to+\infty$ and denote $\eta_n=\eta_{h_n}$, $\thet_n=\thet_{h_n}$. The sequence can be chosen such that $\eta_n\to\eta$ a.e. in $(0,T)\times\Omega$ and there is an integrable majorant of $\{\thet_n\}$ on $(0,T)\times\Omega$. We denote it $\Theta$. Thanks to the Jensen inequality, we have  $\exp \eta_n\leq \thet_n$. 
We set $\zeta:=\psi\T_k(\exp(\etan)-\tmax)_+$ in~\eqref{entropy-limit2}, which is an admissible test function. Our goal is to let $n\to +\infty$ in all terms.  Note that $\nabla(\T_k(\exp(\etan)-\tmax)_+)=\nabla\etan\exp(\etan)\chi_{\{\exp(\etan)\in(\tmax,\tmax+k)\}}\in L^2((0,T)\times\Omega; \R^3)$ and 
\begin{equation}\label{eq:conv-stek}
  |\psi\T_k(\exp(\etan)-\tmax)_+|\leq k\nm{\psi}_\infty,\quad \psi\T_k(\exp(\etan)-\tmax)_+\to \psi\T_k(\thet-\tmax)_+\quad\mbox{a.e. in $\QT$.}
\end{equation}

We introduce  an auxiliary function  
$$
\Fk(s)= \int_{\tmax}^s \frac{\T_k(\sigma-\tmax)_+}{\sigma}\dd s
$$
to deal with the time derivative term. We  compute
$$
\intTO{\dert\etan\exp(\etan)\frac{\T_k(\exp(\etan)-\tmax)_+}{\exp(\etan)}\psi}=-\intTO{\Fk(\exp(\etan))\dert\psi}.
$$
Since $0\leq\Phi_k(s)\leq s$ for $s>0$, we have an integrable majorant for $\Phi_k(\exp(\etan))$, namely
$$
\Phi_k(\exp(\etan))\leq\Phi_k(\Theta)\leq\Theta,
$$
and we can identify the limit as $n\to+\infty$:
$$
-\intTO{\Fk(\exp(\etan))\dert\psi}\to-\intTO{\Fk(\thet)\dert\psi}.
$$

Similarly, we can also pass to the limit $n\to+\infty$ in convective term and all terms of the right hand side. We use \eqref{eq:conv-stek} and the fact that 
$$
\nabla\eta \vv, \qquad  \frac{\S:\D{\vv}}{\thet}, \qquad \kappa(\thet)\,|\nabla \eta|^2
$$
are integrable quantities over $\QT$ by \eqref{MBap}. Consequently, we may assume that (up to a subsequence) their Steklov averages converge to them almost everywhere in $\QT$, and that there exist $L^1(\QT)$-integrable majorants.

\begin{equation}\label{eq:lim-stek1}
\begin{gathered}  \intTO{\left(-\nabla\eta\cdot\vv+\frac{\S:\D{\vv}}{\thet}+\kappa(\thet)\,|\nabla \eta|^2\right)_{h_n}\psi\T_k(\exp(\etan)-\tmax)_+}\\
  \to
  \intTO{\left(-\nabla\eta\cdot\vv+\frac{\S:\D{\vv}}{\thet}+\kappa(\thet)\,|\nabla \eta|^2\right)\psi\T_k(\thet-\tmax)_+}.
\end{gathered}
\end{equation}
It remains to pass to the limit in the elliptic term. For this it is sufficient to realize that 
$$
\psi\T_k(\exp(\etan)-\tmax)_+ \to \psi\T_k(\thet-\tmax)_+ \quad \text{as } n\to+\infty \text{ in } L^2(\QT),
$$
and 
$$
\psi\nabla(\T_k(\exp(\etan)-\tmax)_+) \rightharpoonup \chi \quad \text{in } L^2(\QT) \text{ up to a subsequence}.
$$
Clearly, $\chi = \psi\nabla(\T_k(\thet-\tmax)_+)$, so
$$
\intTO{(\kappa(\thet)\nabla\eta)_{h_n}\cdot \psi\nabla(\T_k(\exp(\etan)-\tmax)_+)} \to
\intTO{\kappa(\thet)\nabla\eta\cdot \psi\nabla(\T_k(\thet-\tmax)_+)}.
$$
After the limit passage we rearrange some terms. The convective term disappears:
$$
\intTO{-\nabla\eta\cdot\vv\psi\T_k(\thet-\tmax)_+} =
\intTO{-\nabla\Psi(\thet)\cdot\vv\psi} = 0,
$$
since $\diver \vv=0$ in $\QT$. Note here, that $\Psi$ is a properly chosen function. 
Let us now combine two terms connected with the elliptic term
\begin{equation}
  \begin{aligned}
  &\intTO{\kappa(\thet)\,|\nabla \eta|^2\psi\T_k(\thet-\tmax)_+}-\intTO{\kappa(\thet)\nabla\eta\cdot\psi\nabla(\T_k(\thet-\tmax)_+)}=\\
  &\phantom{MMM}\intTO{\psi\kappa(\thet)\nabla \eta\cdot\left(\nabla\eta\T_k(\thet-\tmax)_+-\nabla(\T_k(\thet-\tmax)_+)\right)}=\\
  &\phantom{MMM}\intTO{\psi\kappa(\thet)|\nabla \eta|^2\left(\T_k(\thet-\tmax)_+-\thet\chi_{\{\thet-\tmax\in(0,k)\}}\right)}.\\
\end{aligned}
\end{equation}
Analyzing the large bracket above we see that it is equal to $0$ for $\thet\leq\tmax$, equal to $-\tmax$ if $\thet \in(\tmax, \tmax+k)$ and equal to $k$ for $\thet\geq \tmax+k$. Collecting the computation above and letting $\psi$ to $\chi_{(0,T)}$, we deduce that for almost all $T>0$
\begin{equation}\label{limit-k}
\intTO{k\kappa(\thet)|\nabla \eta|^2\chi_{\{\thet>\tmax+k\}}}\leq
\intTO{\tmax\kappa(\thet)|\nabla \eta|^2}+\intO{\Phi_k(\thet)(T)}\leq C,
\end{equation}
where the last inequality is a consequence of ~\eqref{MBap}. Note that the constant $C>0$ is independent of $k$. Moreover, it can be deduced from \eqref{MBap} and \eqref{limit-k} that, see \cite[Section 3.5, (3.20)]{AbBuKa2024},
\begin{equation}\label{eq:reg-grad-thet}
  \intTO{\frac{|\nabla \thet|^2}{\thet^{1+\varepsilon}}}\le C(T,\varepsilon)
\end{equation}
for all $\varepsilon >0$. From this inequality together with \eqref{MBap} one gets for all $T$, see \cite[Appendix~B.5]{BuFeMa2009},
\begin{equation}\label{eq:reg-thet1}
  \thet\in L^{r}(Q_T),\quad\nabla\thet\in L^q(Q_T)\quad\mbox{for $r\in[1,\frac53)$, $q\in[1,\frac54)$.}
\end{equation}
The last property needed for the stability is the interpolation for $\vv$, and from \eqref{MBap} one obtains
\begin{equation}\label{eq:reg-v1}
  \vv \in L^{\frac{5p}{3}}(Q_T).
\end{equation}
We also recall that, directly from the definition of \bws, it follows
%
\begin{equation}\label{ap-theta2}
\begin{split}
\thet(t,x)\ge \tmin.
\end{split}
\end{equation}
This is very important since it gives that $\thet^a$ is bounded in $Q$ for any $a\leq0$.

Finally, we show that for all $k\in \mathbb{N}$ and all $\varepsilon\in (0,1)$
\begin{equation}
\label{rentime}
\partial_t \mathcal{T}_{k,\varepsilon}(\thet) \in \mathcal{M}(0,T; (W^{2,3}_0(\Omega))^*).
\end{equation}
We do not provide the complete rigorous proof, since it essentially follows the previous steps. Hence, we set $\zeta := \exp(\eta_h) \, \mathcal{T}_{k,\varepsilon}'(\exp(\eta_h)) \, \psi$ in~\eqref{entropy-limit2}, where $\psi \ge 0$ is a smooth compactly supported function. Doing so, we have
\begin{equation*}
\begin{split}
&\intTO{\dert \mathcal{T}_{k,\varepsilon}(\exp(\eta_h)) \, \psi} + \intTO{(\nabla\eta \cdot \vv)_h \, \zeta} + \intTO{(\kappa(\thet) \nabla\eta)_h \cdot \nabla \zeta} \\
& \geq \intTO{\left(\frac{1}{\thet}\, \S : \D{\vv}\right)_h \, \zeta} + \intTO{\left(\kappa(\thet) |\nabla\eta|^2 \right)_h \, \zeta}.
\end{split}
\end{equation*}
Letting $h \to 0$ with help of integration by parts in the first two integrals, we deduce (using the above estimates \eqref{eq:reg-grad-thet}--\eqref{eq:reg-v1})
\begin{equation}\label{entropy-limit2T}
\begin{split}
&-\intTO{\mathcal{T}_{k,\varepsilon}(\thet) \, \dert \psi} - \intTO{\mathcal{T}_{k,\varepsilon}(\thet) \, \nabla \psi \cdot \vv} \\
& + \intTO{\kappa(\thet) |\nabla\thet|^2 \, \mathcal{T}''_{k,\varepsilon}(\thet) \, \psi} + \intTO{\kappa(\thet) \nabla \mathcal{T}_{k,\varepsilon}(\thet) \cdot \nabla \psi} \\
& - \intTO{\S : \D{\vv} \, \mathcal{T}'_{k,\varepsilon}(\thet) \, \psi} \geq 0.
\end{split}
\end{equation}
We see that the above expression defines a nonnegative distribution on $Q$. Consequently, there exists a measure $\mu \in \mathcal{M}(Q)$ such that, for all $\psi \in \mathcal{C}^1_0(Q)$, there holds (see \cite[Theorem 9.13]{Leoni09})
\begin{equation}\label{entropy-limit2TE}
\begin{split}
&-\intTO{\mathcal{T}_{k,\varepsilon}(\thet) \, \dert \psi} - \intTO{\mathcal{T}_{k,\varepsilon}(\thet) \, \nabla \psi \cdot \vv} \\
& + \intTO{\kappa(\thet) |\nabla\thet|^2 \, \mathcal{T}''_{k,\varepsilon}(\thet) \, \psi} + \intTO{\kappa(\thet) \nabla \mathcal{T}_{k,\varepsilon}(\thet) \cdot \nabla \psi} \\
& - \intTO{\S : \D{\vv} \, \mathcal{T}'_{k,\varepsilon}(\thet) \, \psi} = \langle \mu, \psi \rangle.
\end{split}
\end{equation}
From the classical Sobolev embedding and the above identity, we can thus deduce \eqref{rentime}.

\subsubsection{Uniform $L^\infty(0,+\infty;L^1(\Omega))$ estimate of temperature
}\label{sss:324}
Note that for a general \bws\ we know that it lies for any $T>0$ in $L^\infty(0,T;L^1(\Omega))$, however we do not know any qualitative estimate.  We need an estimate of $\thet$ in $L^\infty(0,+\infty;L^1(\Omega))$ that is uniform with respect to the initial and boundary conditions. We are going to show that it follows from \eqref{TE} and \eqref{entropy-limit}.

Since $\kappa$ is bounded from above and below, the function $\G$, defined in \eqref{DFG}, is strictly increasing and enjoys the following estimate
\begin{equation}\label{bound-G}
\underline{\kappa} s\leq \G(s) \leq \overline{\kappa} s  \qquad \textrm{ for all } s\ge 0.
\end{equation}
We define for any  $k\in \N$ and any $\alpha\in(0,1)$
$$
\mathcal{F}_k(s):=\int_1^s \frac{\mathcal{T}_k(z)}{z} \, {\rm d}z \ \mbox{ and } \ \mathcal{F}_k^\al(s):=\int_1^s \frac{\mathcal{T}_k(z)}{z} (\G(\mathcal{T}_k(z)))^{-\al}\,{\rm d}z.
$$
It directly follows from the definition that we have the following convergence results
\begin{equation}\label{Fmk}
\begin{aligned}
\mathcal{F}_k(s) &\to s-1 &&\mbox{ for any } s>0, \mbox{ as } k\to +\infty,\\
\mathcal{F}_k^\al(s) &\to \mathcal{H}^\al(s)&&\mbox{ for any } s>0, \mbox{ as } k\to +\infty,
\end{aligned}
\end{equation}
$\mathcal{H}^\al$ being defined in \eqref{DFG}.

Let $\varphi=\varphi(x)$ be a function in $\mathcal{C}^1(\overline\Omega)$ such that $\varphi\in[0,1]$ in $\Omega$ and $\varphi=1$ on $\partial \Omega$. Let $T>0$ be a Lebesgue point of the mappings $\thet,\eta\in L^\infty_{loc}([0,+\infty); L^1(\Omega))$ and $\vv\in L^\infty(0,+\infty;L^2(\Omega))$, and let $\psi=\psi(t)$ be a nonnegative function in $\mathcal{C}^{0,1}([0,T])$ with $\psi(0)=\psi(T)=0$. Let us choose the test function in \eqref{entropy-limit2} as
$$
\zeta:=(1-\varphi) \mathcal{T}_k(\exp{\eta_h})\psi.
$$
Note that such function belongs to $L^\infty((0,T)\times\Omega)\cap L^2(0,T;W^{1,2}_0(\Omega))$ and can be used as a test function in \eqref{entropy-limit2} to get
\begin{multline}\label{eq:pi}
  \intTO{-\cF_k(\exp{\eta_h}) (1-\varphi)\dert\psi}+\intTO{(\nabla\eta\cdot \vv)_h (1-\varphi) \mathcal{T}_k(\exp{\eta_h})\psi}\\  + \intTO{(\kappa(\thet)\nabla\eta)_h \cdot(-\nabla\varphi \mathcal{T}_k(\exp{\eta_h})+(1-\varphi) \chi_{\{\exp{\eta_h}<k\}}\exp{\eta_h}\nabla\eta_h)\psi}\\
  -\intTO{\left(\kappa(\thet)|\nabla\eta|^2\right)_h(1-\varphi) \mathcal{T}_k(\exp{\eta_h})\psi}
  \geq 0.
\end{multline}
We used the fact that the term $(\S:\D{\vv}/\thet)_h$ on the right hand side of \eqref{entropy-limit2} and the test function are positive. Now we choose a sequence $\{h_n\}$ such that $h_n>0$, $h_n\to 0$ as $n\to+\infty$ and denote $\eta_n=\eta_{h_n}$, $\thet_n=\thet_{h_n}$. The sequence can be chosen such that $\eta_n\to\eta$ a.e. in $\QT$ and there is an integrable majorant of $\{\thet_n\}$ on $\QT$. We call it $\Theta$. Thanks to the Jensen inequality, we have  $\exp \eta_n\leq \thet_n$. Consequently, $\cF_k(\exp{\eta_n})\to\cF_k(\thet)$ as $n\to+\infty$ a.e. in $(0,T)\times\Omega$ and $\cF_k(\exp{\eta_n})\leq \Theta$. Dominated Convergence Theorem allows us to pass to the limit as $n\to+\infty$ in the first term on the left hand side of \eqref{eq:pi}. Similarly we can treat also other terms in that inequality to get
\begin{equation*}
\begin{aligned}
  &\intTO{-\cF_k(\thet)(1-\varphi)\dert\psi}+\intnO{\nabla\eta\cdot \vv (1-\varphi) \mathcal{T}_k(\thet)\psi}\\  
  &\quad + \intTO{\kappa(\thet)\nabla\eta \cdot(-\nabla\varphi \mathcal{T}_k(\thet))\psi}+\intnO{\kappa(\thet)\nabla\eta \cdot (1-\varphi) \chi_{\{\thet<k\}}\thet\nabla\eta\psi}\\ 
  &\quad  -\intTO{\kappa(\thet)|\nabla\eta|^2(1-\varphi) \mathcal{T}_k(\thet)\psi}\geq 0.
\end{aligned}
\end{equation*}
Now, we combine the last two terms
$$
\intTO{\kappa(\thet)|\nabla\eta|^2(1-\varphi)\left( \chi_{\{\thet<k\}}\thet
-\mathcal{T}_k(\thet)\right)\psi}\leq 0.
$$
After some reformulation of the second and the third term we arrive to the inequality
$$
  \intTO{-\cF_k(\thet)(1-\varphi)\dert\psi+\cF_k(\thet)\vv\cdot \nabla\varphi\psi - \kappa(\thet)\nabla\thet \cdot\nabla\varphi \frac{\mathcal{T}_k(\thet)}\thet\psi}\geq 0.
  $$
  Finally, we set $\psi=\psi_n(t)=\max(0,\min(1,nT/2-n|t-T/2|))$ for any $n\in\N$ and pass to the limit as $n\to+\infty$, taking into account that $T$ is a Lebesgue point of $\thet\in L^\infty_{loc}([0,T);L^1(\Omega))$ and \eqref{eq:att-ic} we get
\begin{equation}\label{eq:ent-for-l1}
\begin{split}
  &\intO{(1-\varphi)\cF_k(\thet(T))}+\intTO{\cF_k(\thet)\vv\cdot \nabla\varphi - \kappa(\thet)\nabla\thet \cdot\nabla\varphi \frac{\mathcal{T}_k(\thet)}\thet}\\
  &\quad \geq \intO{(1-\varphi)\cF_k(\thet_0)}.
  \end{split}
\end{equation}

We used that $\cF_k'\leq 1$ on $(0,+\infty)$ and consequently
$$
n\int_0^{\frac1n}\intO{|\cF_k(\thet)-\cF_k(\thet_0)|(1-\varphi)}\dt\leq
n\int_0^{\frac1n}\intO{|\thet-\thet_0|(1-\varphi)}\dt\to 0
$$
as $n\to+\infty$. Analogously, it holds that
$$
n\int_{T-\frac1n}^T\intO{|\cF_k(\thet)-\cF_k(\thet(T))|(1-\varphi)}\dt\leq
n\int_{T-\frac1n}^T\intO{|\thet-\thet(T)|(1-\varphi)}\dt\to 0
$$
as $n\to+\infty$.

We realize that one can insert the same $\psi=\psi_n$ also to \eqref{TE} and pass to the limit $n\to+\infty$. Note that $B$ is Lipschitz continuous in the first variable.  Subtracting \eqref{eq:ent-for-l1} from the so obtained inequality yields
\begin{equation}\label{TE-for-l1}
\begin{array}{l}\displaystyle\vspace{6pt}
  \intO{\frac12{|\vv(T)|^2} + \thet(T) -\cF_k(\thet(T))+\left(\cF_k(\thet(T))- \B(\thet(T), \thetah)\right)\varphi}\\ \displaystyle\vspace{6pt}
  -  \intTO{\vv\cdot\nabla \varphi\cF_k(\thet)}-\intTO{\vv\cdot\nabla \thet\bb(\thet, \thetah)\varphi}
  \\ \displaystyle\vspace{6pt}
  + \intTO{\kappa(\thet)\nabla\thet\cdot \nabla \varphi \left(\frac{\mathcal{T}_k(\thet)}\thet -\bb(\thet, \thetah)\right)}
  \\ \displaystyle\vspace{6pt}
  -\intTO{\kappa(\thet)\nabla\thet\cdot\nabla\thetah \,\partial_{2} \bb(\thet, \thetah)\,\varphi}\\\displaystyle\vspace{6pt}
  \leq
  \intO{\left[\frac12{|\vv_0|^2} + \thet_0 -\cF_k(\thet_0)+\left(\cF_k(\thet_0)- \B(\thet_0, \thetah)\right)\varphi\right]}
\end{array}
\end{equation}
Let us consider $\varphi_n:\R^3\to\R$, $n\in\N$ a sequence of functions in $C^1(\R^3)$ such that
$\varphi_n\to 0$ as $n\to+\infty$ a.e. in $\Omega$ and for any $n\in\N$
\begin{equation*}\begin{split}
\varphi_n=1\mbox{ on }\partial\Omega,\quad 0\leq\varphi_n\leq1 \mbox{ a.e. in } \Omega,\quad \ {\rm supp }\ \varphi_n \subseteq \mathcal{U}(\partial\Omega, {1}/{n}) \mbox{ and } |\nabla \varphi_n(x)|\leq Cn
\end{split}\end{equation*}
where $\mathcal{U}(\partial\Omega, {1}/{n})$ denotes the open neighborhood of $\partial\Omega$ of radius ${1}/{n}$.
We want to pass to the limit as $n\to+\infty$ in \eqref{TE-for-l1}. The terms with $\varphi_n$ disappear in the limit since
$$
\cF_k(\thet(T))- \B(\thet(T), \thetah),\cF_k(\thet_0)- \B(\thet_0, \thetah)\in L^1(\Omega), \vv\cdot\nabla \thet \, b(\thet, \thetah), \kappa(\thet)\nabla\thet\cdot\nabla\thetah \,\partial_{2} \bb(\thet, \thetah)\in L^1(Q),
$$
compare the computation in Section~\ref{sec:lvteb}. Further goal is to show that also the terms with $\nabla\varphi_n$ disappear in the limit as $n\to+\infty$. The main tool for this is Hardy's inequality, see Theorem~\ref{thm:hardy}. The general strategy is to realize that $|\nabla \varphi_n(x)|\leq C/\dist(x,\partial\Omega)$ for $x\in\Omega$ and then use a fact that some of the appearing functions lie in some Sobolev space $W^{1,r}_0$ with $r\in(1,+\infty)$.
We start with the second integral on the left hand side of \eqref{TE-for-l1}. 
 Since $\vv \in L^p(0,T;W^{1,p}_0(\Omega))$ we need that $\cF_k(\thet)\in L^{p'}((0,T)\times\Omega)$. This follows from \eqref{eq:reg-thet1} because for fixed $k\in\N$ the function $\cF_k$ has logarithmic growth at infinity. We compute
\begin{multline}
\left|\intTO{\cF_k(\thet)(\vv  \cdot \nabla\varphi_n)\psi }\right|\\
\leq
C\left(\int_0^T\int_{\Omega_{1/n}}{\cF_k(\thet)^{p'}}\dx \dt\right)^{\frac1{p'}}\left(\int_0^T\int_{\Omega}|\vv|^p |\nabla\varphi_n|^p\dx\dt\right)^{\frac1p}\\
\leq
C\left(\int_0^T\int_{\Omega_{1/n}}{\cF_k(\thet)^{p'}}\dx \dt\right)^{\frac1{p'}}\left(\int_0^T\int_{\Omega} \left|\frac{\vv}{\dist(x,\partial\Omega)}\right|^p \dx\dt\right)^{\frac1p}\\
\leq C\left(\int_0^T\int_{\Omega_{1/n}}{\cF_k(\thet)^{p'}}\dx \dt\right)^{\frac1{p'}}\left(\int_0^T\int_{\Omega}|\nabla\vv|^p \dx\dt\right)^{\frac1p}
\end{multline}
where we denoted $\Omega_{1/n}=\{x\in\Omega;\dist(x,\partial\Omega)<1/n\}$. The last integral converges to $0$ as $n\to+\infty$ since $\cF_k(\thet)\in L^{p'}((0,T)\times\Omega)$ and $|(0,T)\times\Omega_{1/n}|\to 0$ as $n\to+\infty$.

In order to deal with the fourth integral on the left hand side of \eqref{TE-for-l1} we show that
\begin{equation}\label{reg:third-pre}
F:=\thet^\alpha\left[\frac{\mathcal{T}_k(\thet)}\thet -\bb(\thet, \thetah)\right]\in L^{2}(0,T;W^{1,2}_0(\Omega))
\end{equation}
for $k\geq\tmax$. Having this estimate, we can proceed exactly as above to show that this integral disappears in the limit $n\to+\infty$, since $\kappa(\thet)\nabla \thet/\thet^\alpha\in L^2((0,T)\times\Omega)$. Let us now prove \eqref{reg:third-pre}. The fact that $F$ has zero trace follows from \eqref{item0} and \eqref{MB6}. It remains to show $\nabla F\in L^2(0,T;L^2(\Omega))$. We compute
\begin{align*}
  \nabla F&=\left[(\alpha-1)\thet^{2\alpha-2}\cT_k(\thet)+\thet^{2\alpha-1}\chi_{\{\thet<k\}}-\alpha\thet^{2\alpha-1}\bb(\thet,\thetah)-\thet^{2\alpha}\partial_1b(\thet,\thetah)\right]\thet^{-\alpha}\nabla\thet\\
  &+\left[-\partial_2\bb(\thet,\thetah)\thet^\alpha\right]\nabla\thetah =:F_1\thet^{-\alpha}\nabla\thet+F_2\nabla\thetah.
\end{align*}
The functions $F_1$ and $F_2$ are bounded by \eqref{ap-theta2}, the fact that $\alpha<1$ and properties of the function $b$ listed in Definition~\ref{def:blafnc}, and $\thet^{-\alpha}\nabla\thet,\nabla\thetah\in L^2(Q)$ by \eqref{eq:reg-grad-thet} since $\alpha>1/2$. The statement \eqref{reg:third-pre} follows.
Altogether, setting $\varphi_n$ as a test function in \eqref{TE-for-l1} and passing $n\to+\infty$ we get
\begin{equation}\label{TE-for-l1-fin}
  \intO{\frac12{|\vv(T)|^2} + \thet(T)-1 -\cF_k(\thet(T))}
  \leq \intO{\frac12{|\vv_0|^2} + \thet_0-1 -\cF_k(\thet_0)}.
\end{equation}
Now, we realize that the non-negative function $\Phi(s)=s-1-\cF_k(s)$, $s>0$ satisfies estimates
\begin{equation}\label{est:ofF}
\begin{aligned}
  \Phi(s)\leq (s-k)_+\quad &\mbox{for $s>0$,}\qquad
  \frac s2\leq \Phi(s)+k\lg(2)&\mbox{for $s>2k$.}
\end{aligned}
\end{equation}
The second estimate follows from the facts that $(\Phi(s)-s/2)'\geq 0$ for $s\geq 2k$ and $\Phi(2k)-k=-k\lg(2)$.
Finally, we set $k=\tmax$ and use the estimates \eqref{est:ofF} and \eqref{TE-for-l1-fin} to get
\begin{multline*}
  \intO{\thet(T)}\leq \intO{\thet(T)\chi_{\thet\leq 2\tmax}}+\intO{\thet(T)\chi_{\thet>2\tmax}}
  \leq 2\tmax|\Omega|+\intO{{|\vv(T)|^2} + 2\Phi(\thet(T))+2\tmax\lg(2)}\\
  \leq C\tmax|\Omega|+\intO{{|\vv_0|^2} + 2\Phi(\thet_0)}
  \leq C\intO{\frac12|\vv_0|^2+\max(\tmax,\thet_0)},
\end{multline*}
which gives the desired uniform $L^\infty(0,+\infty;L^1(\Omega))$ estimate
\begin{equation}\label{est:lnl1}
\nm{\thet}_{L^\infty(0,+\infty;L^1(\Omega))}\leq C\intO{\frac12|\vv_0|^2+\max(\tmax,\thet_0)}.
\end{equation}

\subsection{Renormalization of the entropy inequality}\label{SectEntrop}
Further we need to derive an estimate that is suitable for proof of the stability of the steady state. We use a similar method to the one used to derive the estimate \eqref{est:lnl1}. 
Let $\varphi=\varphi(x)$ be again a function in $\mathcal{C}^1(\overline\Omega)$ such that $\varphi\in[0,1]$ in $\Omega$ and $\varphi=1$ on $\partial \Omega$, and let $\psi=\psi(t)$ be a function in $\mathcal{C}^\infty_0(0, +\infty)$, $\psi\geq 0$. Let $T>0$ be such that $\supp\psi\subset[0,T]$. We choose the test function in \eqref{entropy-limit2} as 
$$
\zeta:=(1-\varphi) \mathcal{T}_k(\exp{\eta_h})\, \left(\frac{\G(\thetah)}{\G(\mathcal{T}_k(\exp{\eta_h}))}\right)^\alpha\psi, \mbox{ with $\alpha\in(0,1)$.}
$$
We expand each resulting term and subsequently analyze the limits, first as $h \to 0_+$ and then as $k \to +\infty$, in order to derive a renormalized form of the entropy inequality.  
Throughout, we write $\int_0^{\infty}$, since $\psi$ is compactly supported in $(0,T)$.  
Using integration by parts, together with the compact support of $\psi$ and the fact that $\thetah$ is independent of time, the time-derivative term in \eqref{entropy-limit2} can be rewritten as
\begin{equation*}\begin{split}
&\intnO{\dert\eta_h \, (1-\varphi)\, \mathcal{T}_k(\exp{\eta_h})\,\left(\frac{\G(\thetah)}{\G(\mathcal{T}_k(\exp{\eta_h}))}\right)^{\al} \psi } \\&= \intnO{\dert\left[\left( \mathcal{F}_k^\al(\exp{\eta_h}) - \varphi\mathcal{F}_k^\al(\exp{\eta_h})\right)(\G(\thetah))^{\al}\right]\psi}\\
&= -\intnO{\left[\left( \mathcal{F}_k^\al(\exp{\eta_h}) - \varphi\mathcal{F}_k^\al(\exp{\eta_h})\right)(\G(\thetah))^{\al}\right]\dert\psi}.
\end{split}\end{equation*}
 Thanks to the Jensen inequality, we have  $\exp \eta_h\leq \thet_h$. As $\thet\in L^{1}(Q_T)$ then $\thet_h \to \thet$ in $L^1(Q_T)$ and consequently also $\exp \eta_h \to \exp \eta =\thet$ in $L^1(Q_T)$. As
\begin{equation}\label{est:forleb}
  |\mathcal{F}_k^{\al}(s)| \le C(1+s)\quad\mbox{for $s>0$,}
\end{equation}
we can use the Dominated Convergence Theorem to deduce that (recall also that $\thetah$ is bounded)
\begin{equation*}
\begin{split}
\lim_{h\to 0_+} &\left[ -\intnO{\left[\left( \mathcal{F}_k^\al(\exp{\eta_h}) - \varphi\mathcal{F}_k^\al(\exp{\eta_h})\right)(\G(\thetah))^{\al}\right]\dert\psi}\right]\\
& = - \intnO{\left[\left( \mathcal{F}_k^\al(\thet) - \varphi\mathcal{F}_k^\al(\thet)\right)(\G(\thetah))^{\al}\right]\dert\psi}.
\end{split}
\end{equation*}
As \eqref{est:forleb} is actually for fixed $\alpha$ uniform in $k\in\mathbb{N}$ we can use \eqref{Fmk} and  
again the Dominated Convergence Theorem to get
\begin{equation*}\begin{split}
&\lim_{k\to +\infty} \left[-\intnO{\left[\left( \mathcal{F}_k^\al(\thet) - \varphi\mathcal{F}_k^\al(\thet)\right)(\G(\thetah))^{\al}\right]\dert\psi}\right]\\
& = - \intnO{[\mathcal{H}^\al(\thet) - \varphi\,\mathcal{H}^\al(\thet)](\G(\thetah))^{\al}\, \dert\psi}.
\end{split}\end{equation*}
Next, we focus on the term with $\S:\D \vv$. First, we consider the limit $h\to 0_+$ of
$$
\intnO{\left(\frac{1}{\thet}\,\S:\D{\vv}\right)_h\, \mathcal{T}_k(\exp{\eta_h}) \left(\frac{\G(\thetah)}{\G(\mathcal{T}_k(\exp{\eta_h}))}\right)^{\al}(1- \varphi)\,\psi}.
$$
Repeating the arguments for $\exp \eta_h$ already performed, also using that $\frac{1}{\thet}\,\S:\D{\vv}\in L^1((0,+\infty)\times\Omega)$, see \eqref{3.1.5} and \eqref{ap-theta2}, we are in position to employ the Dominated Convergence Theorem and take the limit under the integral sign, then
\begin{equation}\label{SD}\begin{split}
 &\intnO{\frac{1}{\thet}\,\S:\D{\vv}\, \mathcal{T}_k(\thet)\left(\frac{\G(\thetah)}{\G(\mathcal{T}_k(\thet))}\right)^{\al} (1- \varphi)\,\psi}
\\& =\intnO{\chi_{\{\thet\leq k\}}\, \S:\D{\vv}\, \left(\frac{\G(\thetah)}{\G(\thet)}\right)^{\al} (1- \varphi)\,\psi}\\&  + \intnO{\chi_{\{\thet> k\}}\,\frac{k}{\thet}\,\S:\D{\vv}\, \left(\frac{\G(\thetah)}{\G(k)}\right)^{\al} (1- \varphi)\,\psi}.
\end{split} \end{equation}
Using that $\|\thetah\|_\infty\leq k$ for $k$ sufficiently large, the last integral converges to zero when $k\to +\infty$. Thus taking the limit as $k\to +\infty$ in \eqref{SD} we get
$$
\intnO{  \S:\D{\vv}\,  \left(\frac{\G(\thetah)}{\G(\thet)}\right)^{\al} (1- \varphi)\,\psi}.
$$
Now we focus on the terms
\begin{equation*}
\begin{split}
&\intnO{(\kappa(\thet)\nabla\eta)_h \cdot\nabla \left[ \mathcal{T}_k(\exp{\eta_h})(1-\varphi) \left(\frac{\G(\thetah)}{\G(\mathcal{T}_k(\exp{\eta_h}))}\right)^{\al} \right]\psi } \\
&- \intnO{\left(\kappa(\thet)\frac{|\nabla\thet|^2}{\thet^2}\right)_h\, \mathcal{T}_k(\exp{\eta_h})(1- \varphi) \left(\frac{\G(\thetah)}{\G(\mathcal{T}_k(\exp{\eta_h}))}\right)^{\al} \psi}.
\end{split}
\end{equation*}
Since $\kappa(\thet)$ is bounded, $\nabla\eta\in L^2((0, T)\times\Omega)$, $\mathcal{T}_k(\exp{\eta_h})$ is uniformly bounded in $h>0$ and $\eta_h$ converges almost everywhere to $\eta$, we can take the limit as $h\to 0_+$ under the integral sign. After some manipulations we get
\begin{equation*}
\begin{split}
&\intnO{\nabla \G(\thet)\cdot \left[\frac{\nabla (\mathcal{T}_k(\thet))}{\thet} -\frac{\mathcal{T}_k(\thet)\nabla\thet}{\thet^2}\right] (1-\varphi) \left(\frac{\G(\thetah)}{\G(\mathcal{T}_k(\thet))}\right)^{\al} \psi} \\
&+\al\intnO{\nabla \G(\thet) \cdot \nabla\left[ \frac{\G(\thetah)}{\G(\mathcal{T}_k(\thet))}\right]\left( \frac{\G(\thetah)}{\G(\mathcal{T}_k(\thet))}\right)^{\al-1}\frac{\mathcal{T}_k(\thet)}{\thet}\,(1-\varphi)\,\psi}\\
&-\intnO{\nabla \G(\thet)\cdot \nabla\varphi  \left(\frac{\G(\thetah)}{\G(\mathcal{T}_k(\thet))}\right)^{\al}\frac{\mathcal{T}_k(\thet)}\thet \psi} =:A_1+A_2+A_3.
\end{split}
\end{equation*}
Now employing \eqref{limit-k} we get that
$$
|A_1|\leq C\intTO{\chi_{\{\thet>k\}}k\frac{|\nabla\thet|^2}{\thet^2}\left(\frac{\G(\thetah)}{\G(k)}\right)^{\al}}\to 0\quad\mbox{as $k\to+\infty$.} 
$$
Next, $A_2$ can be split as
\begin{equation}\label{B}
\begin{split}
A_2&= \al \intnO{\chi_{\{\thet\leq k\}} \nabla \G(\thet) \cdot \nabla\left[ \frac{\G(\thetah)}{\G(\mathcal{T}_k(\thet))}\right]\left( \frac{\G(\thetah)}{\G(\mathcal{T}_k(\thet))}\right)^{\al-1}\,(1-\varphi)\,\psi} \\
 &+\al\intnO{\chi_{\{\thet>k\}}{\nabla \G(\thet) \cdot \nabla\left[ \frac{\G(\thetah)}{\G(k)}\right]\left( \frac{\G(\thetah)}{\G(k)}\right)^{\al-1}\frac{k}{\thet}\,\,(1-\varphi)\,\psi }}=:A_{21}+A_{22}.
\end{split}
\end{equation}
Let us observe that the second integral vanishes as $k\to +\infty$. Indeed by a simple manipulation and the H\"{o}lder inequality we get that
\begin{equation*}
|A_{22}|\leq Ck^{\frac12-\al}\left(\intTO{\chi_{\{\thet>k\}} \frac{k|\nabla\thet|^2}{\thet^{2}}}\right)^{\frac{1}{2}} \left( \intTO{\chi_{\{\thet>k\}}|\nabla\thetah|^2}\right)^{\frac{1}{2}}.
\end{equation*}
Recalling \eqref{limit-k}, $\alpha>1/2$ and the summability of $|\nabla\thetah|^2$ we get the convergence to zero as $k\to +\infty$.
It remains to discuss the first integral term in \eqref{B}, which can be reformulated onto the whole domain $(0, T)\times\Omega$ because $\nabla \G(\mathcal{T}_k(\thet))=0$ when $\thet\geq k$. Through some manipulations and the integration by parts one gets
\begin{equation}\label{ell}
\begin{split}
&A_{21}=-\al \intnO{ \nabla \G(\mathcal{T}_k(\thet)) \cdot \nabla\left[ \frac{\G(\mathcal{T}_k(\thet))}{\G(\thetah)}\right]\left( \frac{\G(\mathcal{T}_k(\thet))}{\G(\thetah)}\right)^{-1-\al}\,(1-\varphi)\,\psi}\\
&=-\al\intnO{ \G(\thetah) \,\nabla\left(\frac{\G(\mathcal{T}_k(\thet))}{\G(\thetah)}\right) \cdot \nabla\left( \frac{\G(\mathcal{T}_k(\thet))}{\G(\thetah)}\right)\left( \frac{\G(\mathcal{T}_k(\thet))}{\G(\thetah)}\right)^{-1-\al}\,(1-\varphi)\,\psi}\\
&\quad - \al \intnO{  \nabla \G(\thetah) \cdot \nabla\left( \frac{\G(\mathcal{T}_k(\thet))}{\G(\thetah)}\right)\left( \frac{\G(\mathcal{T}_k(\thet))}{\G(\thetah)}\right)^{-\al}\,(1-\varphi)\,\psi}\\
&=-\frac{4\al}{(1-\al)^2}\intnO{ \G(\thetah) \left|\nabla\left[\left(\frac{\G(\mathcal{T}_k(\thet))}{\G(\thetah)}\right)^{\frac{1-\al}{2}}\right]\right|^2 \,(1-\varphi)\,\psi}\\
&\quad - \frac{\al}{1-\al} \intnO{  \nabla \G(\thetah) \cdot \left(\left( \frac{\G(\mathcal{T}_k(\thet))}{\G(\thetah)}\right)^{1-\alpha}-1\right)\,\nabla\varphi\,\psi}\\
&\quad - \frac{\al}{1-\al} \intnO{  \nabla \G(\thetah) \cdot \nabla\left[\left(\left( \frac{\G(\mathcal{T}_k(\thet))}{\G(\thetah)}\right)^{1-\alpha}-1\right)\,(1-\varphi)\right]\,\psi}.
\end{split}
\end{equation}
Now, let us observe that the last integral vanishes because $\thetah$ satisfies \eqref{MB6} and for almost all $t\in (0,T)$ the function $(( \G(\mathcal{T}_k(\thet))/\G(\thetah))^{1-\alpha}-1)(1-\varphi)$ belongs to $W^{1,2}(\Omega)$ by \eqref{eq:reg-grad-thet}. It has zero trace provided that $k$ is sufficiently large.
%
Finally, we realize that thanks to the almost everywhere  pointwise convergence  of $\mathcal{T}_k(\thet)$ to $\thet$ and \eqref{eq:reg-grad-thet} we get
$$
\nabla\left[\left(\frac{\G(\mathcal{T}_k(\thet))}{\G(\thetah)}\right)^{\frac{1-\al}{2}}\right]\rightharpoonup
\nabla\left[\left(\frac{\G(\thet)}{\G(\thetah)}\right)^{\frac{1-\al}{2}}\right]\quad\mbox{in $L^2(\QT;\R^3)$.}
$$
Using this in the first integral on the right hand side of~\eqref{ell}  together with the nonnegativity of $1- \varphi$, and consequently of the whole integrand, we can employ weak lower semicontinuity of variational integrals and proceed as in Subsection~\ref{sec:leb} to get
\begin{equation*}
\begin{split}
&\frac{4\al}{(1-\al)^2}\intnO{ \G(\thetah) \left|\nabla\left[\left(\frac{\G(\thet)}{\G(\thetah)}\right)^{\frac{1-\al}{2}}\right]\right|^2 \,(1-\varphi)\,\psi}\\
&\quad \leq \liminf_{k\to +\infty}\frac{4\al}{(1-\al)^2}\intnO{ \G(\thetah) \left|\nabla\left[\left(\frac{\G(\mathcal{T}_k(\thet))}{\G(\thetah)}\right)^{\frac{1-\al}{2}}\right]\right|^2 \,(1-\varphi)\,\psi}.
\end{split}
\end{equation*}
Limit passage in the second term on the right hand side of \eqref{ell} is easy due to $\alpha>1/2$, \eqref{MBap} and the Vitali convergence theorem. 

The Vitali convergence  theorem and \eqref{eq:reg-grad-thet} also yield that $A_3$ converges to 
$$ -\intnO{\nabla \G(\thet)\cdot \nabla\varphi  \left(\frac{\G(\thetah)}{\G(\thet)}\right)^{\al} \psi}.$$

Altogether we have got that
\begin{equation}
\begin{split}
  \liminf_{k\to+\infty}\lim_{h\to 0_+}
   &\intnO{\left(\kappa(\thet)\frac{|\nabla\thet|^2}{\thet^2}\right)_h\, \mathcal{T}_k(\exp{\eta_h})(1- \varphi) \left(\frac{\G(\thetah)}{\G(\mathcal{T}_k(\exp{\eta_h}))}\right)^{\al} \psi}\\
   -&\intnO{(\kappa(\thet)\nabla\eta)_h \cdot\nabla \left[ \mathcal{T}_k(\exp{\eta_h})(1-\varphi) \left(\frac{\G(\thetah)}{\G(\mathcal{T}_k(\exp{\eta_h}))}\right)^{\al} \right]\psi } \\
   \geq
   &\frac{4\al}{(1-\al)^2}\intnO{ \G(\thetah) \left|\nabla\left[\left(\frac{\G(\thet)}{\G(\thetah)}\right)^{\frac{1-\al}{2}}\right]\right|^2 \,(1-\varphi)\,\psi}\\
   +&\frac{\al}{1-\al} \intnO{  \nabla \G(\thetah) \cdot \left(\left( \frac{\G(\thet)}{\G(\thetah)}\right)^{1-\alpha}-1\right)\,\nabla\varphi\,\psi}\\
   +&\intnO{\nabla \G(\thet)\cdot \nabla\varphi  \left(\frac{\G(\thetah)}{\G(\thet)}\right)^{\al} \psi}.
\end{split}
\end{equation}

Finally, we focus on the convective term
$$
\intnO{(\nabla\eta\cdot\vv)_h  \mathcal{T}_k(\exp{\eta_h})\left(\frac{\G(\thetah)}{\G(\mathcal{T}_k(\exp{\eta_h}))}\right)^{\al}(1-\varphi)\,\psi}.
$$
Since $\nabla\eta\cdot\vv\in L^q((0, T)\times\Omega)$ for some $q>1$ by \eqref{MBap}, \eqref{eq:reg-v1} and $p>6/5$, it holds that $(\nabla\eta\cdot\vv)_h$ converges to $\nabla\eta\cdot\vv$ in $L^q((0, T)\times\Omega)$ with $q>1$; moreover $\eta_h$ converges to $\eta$ in any $L^q$ with $q\geq 1$ and $ \mathcal{T}_k(\exp{\eta_h})$ is uniformly bounded in $h>0$. Thus, the Lebesgue dominated convergence theorem ensures that
\begin{equation*}
\begin{split}
&\lim_{h\to 0_+} \intnO{(\nabla\eta\cdot\vv)_h \, \mathcal{T}_k(\exp{\eta_h})\left(\frac{\G(\thetah)}{\G(\mathcal{T}_k(\exp{\eta_h}))}\right)^{\al}(1-\varphi)\,\psi}\\
&=  \intnO{\nabla\eta\cdot\vv \, \mathcal{T}_k(\thet)\left(\frac{\G(\thetah)}{\G(\mathcal{T}_k(\thet))}\right)^{\al}(1-\varphi)\,\psi}.
\end{split}
\end{equation*}
Next, we use the integration by parts and the facts that $\vv$ is divergence-free and $\chi_{\{\thet\leq k\}}\nabla\thet=\nabla\mathcal{T}_k(\thet)$ a.e. in $\QT$ to observe
\begin{equation}\label{ct}
\begin{split}
 \intnO{&\nabla\eta\cdot\vv \, \mathcal{T}_k(\thet)\left(\frac{\G(\thetah)}{\G(\mathcal{T}_k(\thet))}\right)^{\al}(1-\varphi)\,\psi }\\&=\intnO{\nabla \thet \cdot\vv \, \frac{\mathcal{T}_k(\thet)}{\thet} \left(\frac{\G(\thetah)}{\G(\mathcal{T}_k(\thet))}\right)^{\al} (1-\varphi)\,\psi }\\&=\intnO{\chi_{\{\thet>k\}}\nabla \thet \cdot\vv \,\frac{k }{\thet} \left(\frac{\G(\thetah)}{\G(k)}\right)^{\al} (1-\varphi)\,\psi } \\
& -\intnO{\mathcal{T}_k(\thet)\,\vv  \cdot \nabla\left[\left(\frac{\G(\mathcal{T}_k(\thet))}{\G(\thetah)}\right)^{-\al}\right] (1-\varphi)\,\psi } \\&+ \intnO{\mathcal{T}_k(\thet)\,\vv  \cdot \nabla\varphi\left(\frac{\G(\thetah)}{\G(\mathcal{T}_k(\thet))}\right)^{\al} \,\psi } .
\end{split}
\end{equation}
The last term converges by \eqref{eq:reg-v1} and \eqref{eq:reg-thet1} to
$$  \intnO{\thet\,\vv  \cdot \nabla\varphi\left(\frac{\G(\thetah)}{\G(\thet)}\right)^{\al} \,\psi } $$
provided $\alpha>(3-2p)/(3p)$.
For the first term, we use the assumption \eqref{limit-k}, the H\"{o}lder inequality and the estimate \eqref{bound-G} to deduce
\begin{equation*}
\begin{split}
&\left|\intnO{\chi_{\{\thet>k\}}\vv \cdot \frac{k \nabla \thet}{\thet} \left(\frac{\G(\thetah)}{\G(k)}\right)^{\al}  (1-\varphi)\,\psi }\right| \\
&\qquad \le Ck^{\frac12 - \alpha}\left(\intTO{\chi_{\{\thet>k\}}\frac{k |\nabla \thet|^2}{\thet^2}}
\right)^{\frac{1}{2}} \to 0 \quad \textrm{ as } k\to\infty,
\end{split}
\end{equation*}
since $\al>\frac12$. For the second term on the right hand side, we use the Young inequality, \eqref{bound-G} and properties of $\thetah$ to conclude
\begin{equation}\label{ct2}
\begin{split}
&\left|\intnO{\mathcal{T}_k(\thet)\,\vv  \cdot \nabla\left[\left(\frac{\G(\mathcal{T}_k(\thet))}{\G(\thetah)}\right)^{-\al}\right](1-\varphi)\,\psi}\right|\\
&\quad \le \frac{\alpha}{2}\intnO{\mathcal{G}(\thetah)\left(\frac{\G(\mathcal{T}_k(\thet))}{\G(\thetah)}\right)^{-1-\al}\left|\nabla\frac{\G(\mathcal{T}_k(\thet))}{\G(\thetah)}
\right|^2 (1-\varphi)\,\psi}\\
&\qquad +C\intnO{\mathcal{T}_k(\thet)^{1-\al}|\vv|^2 (1-\varphi)\,\psi}.
\end{split}
\end{equation}
We observe here that the first term on the right hand side will be absorbed by the first term on the right hand side of \eqref{ell} while we can pass to the limit as $k\to+\infty$ in the second term on the right hand side of \eqref{ct2} by Lebesgue Monotone Convergence Theorem. Finally we want to mention that the limit integral
$$
\intnO{{\thet^{1-\al}|\vv|^2(1-\varphi)\,\psi}}
$$
converges for $\alpha>(6-2p)/(3p)$ by \eqref{eq:reg-grad-thet} and \eqref{eq:reg-v1}. Note that for $p=6/5$ we get the condition $\alpha>1$ and for $p=3/2$ we get $\alpha>2/3$.

Collecting all the terms we get for $1>\alpha>\max(1/2,(6-2p)/3p)$
\begin{equation}\label{En-Final}
\begin{split}
& \intnO{[\mathcal{H}^\al(\thet) - \varphi\,\mathcal{H}^\al(\thet)]\G(\thetah)^{\al}\, \dert\psi}\\
&\qquad +\frac{2\al}{(1-\al)^2}\intnO{ \G(\thetah) \left|\nabla\left(\frac{\G(\thet)}{\G(\thetah)}\right)^{\frac{1-\al}{2}}\right|^2 \,(1-\varphi)\,\psi} \\
&\qquad + \intnO{  \S:\D{\vv}\,  \left(\frac{\G(\thetah)}{\G(\thet)}\right)^{\al} (1- \varphi)\,\psi}\\
&\leq
C\intnO{\thet^{1-\al}|\vv|^2 (1-\varphi)\,\psi} -\intnO{\nabla \G(\thet)\cdot \nabla\varphi  \left(\frac{\G(\thetah)}{\G(\thet)}\right)^{\al} \psi}\\
&\qquad 
-\frac{\alpha}{1-\alpha}\intnO{\nabla \G(\thetah)\cdot \nabla\varphi  \left[\left(\frac{\G(\thet)}{\G(\thetah)}\right)^{1-\al}-1\right] \psi} \\
&\qquad + \intnO{\thet\,\vv  \cdot \nabla\varphi\left(\frac{\G(\thetah)}{\G(\thet)}\right)^{\al} \,\psi } .
\end{split}\end{equation}

\subsection{Long time behaviour of the temperature}
Let us keep $\psi\in \mathcal{C}_0^\infty(0,+\infty)$ with $\supp\psi\in[0,T]$ for some $T>0$ and consider the same functions $\varphi_n:\R^3\to\R$, $n\in\N$ as in Section~\ref{sss:324}, i.e., a sequence of functions in $\mathcal{C}^1(\R^3)$ such that
$\varphi_n\to 0$ as $n\to+\infty$  a.e. in $\Omega$ and for any $n\in\N$
\begin{equation*}\begin{split}
\varphi_n=1\mbox{ on }\partial\Omega,\quad 0\leq\varphi_n\leq1 \mbox{ a.e. in } \Omega,\quad \ {\rm supp }\ \varphi_n \subseteq \mathcal{U}(\partial\Omega, \frac1n) \mbox{ and } |\nabla \varphi_n(x)|\leq Cn
\end{split}\end{equation*}
where $\mathcal{U}(\partial\Omega, {1}/{n})$ denotes an open neighborhood of $\partial\Omega$ of radius ${1}/{n}$.
Now, let us consider $\varphi=\varphi_n$ with a fixed $n\in\N$ in both \eqref{TE} and \eqref{En-Final} and let us add the resulting inequalities. We get, using the fact that $\intnO{\thet^{1-\al}|\vv|^2\varphi_n\psi}\geq 0$, 
\begin{equation}\label{dec-temp}
\begin{aligned}
     &\intnO{\left[\mathcal{H}^\al(\thet)\G(\thetah)^{\al}-\thet - \frac{|\vv|^2}{2} + \varphi_n(\B(\thet, \thetah)-\mathcal{H}^\al(\thet)(\G(\thetah))^{\al}) \right]\dert\psi} \\
     &\qquad + \frac{2\al}{(1-\al)^2}\intnO{ \G(\thetah) \left|\nabla\left(\frac{\G(\thet)}{\G(\thetah)}\right)^{\frac{1-\al}{2}}\right|^2 \,(1-\varphi_n)\,\psi}\\
&\qquad     + \intnO{  \S:\D{\vv}\,  \left(\frac{\G(\thetah)}{\G(\thet)}\right)^{\al} (1- \varphi_n)\,\psi}\\
&\leq C\intnO{\thet^{1-\al}|\vv|^2 \psi} +\intnO{\thet\left(\frac{\G(\thetah)}{\G(\thet)}\right)^{\al} (\vv  \cdot \nabla\varphi_n)\psi }\\
&\qquad +\intnO{\kappa(\thet)\left[\bb(\thet, \thetah)-\left(\frac{\G(\thetah)}{\G(\thet)}\right)^{\al} \right](\nabla\thet\cdot \nabla \varphi_n)\psi}\\
&\qquad 
-\frac{\alpha}{1-\alpha}\intnO{\nabla \G(\thetah)\cdot \nabla\varphi_n  \left[\left(\frac{\G(\thet)}{\G(\thetah)}\right)^{1-\al}-1\right] \psi}\\
&\qquad  + \intnO{\left(\kappa(\thet)\nabla\thet\cdot\nabla\thetah \,\partial_2 \bb(\thet, \thetah)+\vv\cdot\nabla \thet \,\bb(\thet, \thetah)\right)\varphi_n\, \psi }.
   \end{aligned}
 \end{equation}
We now pass to the limit as $n \to +\infty$. We begin with the second integral, where Fatou's lemma can be applied. After this limit passage, the factor $\varphi_n$ vanishes in this term. The remaining occurrences of $\varphi_n$, namely in parts of the first, third, and last integrals, also disappear in the limit by the Dominated Convergence Theorem, since $\varphi_n \in [0,1]$ and $\varphi_n \to 0$ as $n \to +\infty$, and the following quantities 
\begin{equation}
\begin{aligned}\label{eq:ljedna}
  &\B(\thet, \thetah), \qquad \mathcal{H}^\al(\thet)(\G(\thetah))^{\al},\qquad \S:\D{\vv}\,  \left(\frac{\G(\thetah)}{\G(\thet)}\right)^{\al},\\
  &\kappa(\thet)\nabla\thet\cdot\nabla\thetah \,\partial_2 \bb(\thet, \thetah),\qquad \vv\cdot\nabla \thet \, b(\thet, \thetah)
\end{aligned}
\end{equation}
belong to $L^1(\QT)$.
This  fact for all quantities in \eqref{eq:ljedna} follows from the regularity of the \bws\ and our assumptions on $\alpha$,
see Section~\ref{sec:lvteb} for similar computations.

Let us now consider the terms where $\nabla\varphi_n$ appears. i.e. the second, the third and the fourth integrals on the right hand side of \eqref{dec-temp}. The goal is to show that also these terms disappear in the limit as $n\to+\infty$. The main tool for this is Hardy's inequality, see Theorem~\ref{thm:hardy}. As in Section~\ref{sss:324} the general strategy is to realize that $|\nabla \varphi_n(x)|\leq C/\dist(x,\partial\Omega)$ and then use a fact that some of the appearing function lies in some Sobolev space $W^{1,r}_0$ with $r\in(1,+\infty)$. We start with the second term on the right hand side of \eqref{dec-temp}. 
Since $\vv \in L^p(0,T;W^{1,p}_0(\Omega))$ we need that $F:=\thet({\G(\thetah)}/{\G(\thet)})^{\al} \in L^{p'}(\QT)$. This follows from \eqref{eq:reg-thet1} provided $\al>1-5/(3p')=5/(3p)-2/3$. Then we can compute
\begin{equation}
\begin{aligned}
&\left|\intnO{\thet\left(\frac{\G(\thetah)}{\G(\thet)}\right)^{\al}(\vv  \cdot \nabla\varphi_n)\psi } \right|\\
&\qquad \leq
C\left(\int_0^T\int_{\Omega_{1/n}}{F^{p'}}\dx \dt\right)^{\frac1{p'}}\left(\int_0^T\int_{\Omega}|\vv|^p |\nabla\varphi_n|^p\dx\dt\right)^{\frac1p}\\
&\qquad \leq
C\left(\int_0^T\int_{\Omega_{1/n}}{F^{p'}}\dx \dt\right)^{\frac1{p'}}\left(\int_0^T\int_{\Omega}\left|\frac{\vv}{\dist(x,\partial\Omega)}\right|^p \dx\dt\right)^{\frac1p}\\
&\qquad \leq C\left(\int_0^T\int_{\Omega_{1/n}}{F^{p'}}\dx \dt\right)^{\frac1{p'}}\left(\int_0^T\int_{\Omega}|\nabla\vv|^p \dx\dt\right)^{\frac1p}
\end{aligned}
\end{equation}
where we denoted $\Omega_{1/n}=\{x\in\Omega;\dist(x,\partial\Omega)<1/n\}$. The right hand side term converges to $0$ as $n\to+\infty$ since $F\in L^{p'}((0,T)\times\Omega)$.

In order to deal with the third term on the right hand side we show that
\begin{equation}\label{reg:third}
F:=\thet^r\left[\bb(\thet, \thetah)-\left(\frac{\G(\thetah)}{\G(\thet)}\right)^{\al}\right]\in L^{2}(0,T;W^{1,2}_0(\Omega))
\end{equation}
for some $r>1/2$. Having this estimate we can proceed exactly as above to show that this integral disappears in the limit $n\to+\infty$ since $\kappa(\thet)\nabla \thet/\thet^r\in L^2((0,T)\times\Omega)$ for any $r>1/2$. Let us now prove \eqref{reg:third}. The fact that $F$ has zero trace follows from \eqref{item0}. It remains to show $\nabla F\in L^2(0,T;L^2(\Omega))$. We compute
\begin{align*}
  \nabla F&=\nabla\thet\left[r\thet^{r-1}\big(\bb(\thet,\thetah)-\left(\frac{\G(\thetah)}{\G(\thet)}\right)^{\al}\big)+\thet^r\big(\partial_1\bb(\thet,\thetah)+\left(\frac{\G(\thetah)}{\G(\thet)}\right)^{\al}\frac{\al\kappa(\thet)}{\G(\thet)}\big)\right]\\
  &+\nabla\thetah\left[\partial_2\bb(\thet,\thetah)-\left(\frac{\G(\thetah)}{\G(\thet)}\right)^{\al}\frac{\al\kappa(\thetah)}{\G(\thetah)}\right]\thet^r=:F_1\nabla\thet+F_2\nabla\thetah.
\end{align*}
The function $F_2$ satisfies the estimate $|F_2|\leq C\thet^{r-\alpha}$ by \eqref{item2}, \eqref{MB6} and \eqref{bound-G} so it suffices to choose $r\in(1/2,\alpha)$. We get that $F_2$ is bounded and since $\nabla\thetah\in L^2(\Omega)$, see \eqref{MB6}, we obtain $F_2\nabla\thetah\in L^2((0,T)\times\Omega)$. Similarly, $F_1\leq C\thet^{r-1-\alpha}$ by \eqref{item0}, \eqref{item3}, \eqref{MB6}, \eqref{MBBounds} and \eqref{bound-G}. If $r-1-\al<-1/2$ then $F_1\nabla\thet\in L^2((0,T)\times\Omega)$ by \eqref{eq:reg-grad-thet}. It suffices to choose $r\in(1/2, 1/2 +\alpha)$. 

The fourth term on the right hand side of \eqref{dec-temp} we rewrite as
$$
\intnO{\frac{\nabla \G(\thetah)}{\G(\thetah)^{1-\al}}\cdot \nabla\varphi_n  [\G(\thet)^{1-\al}-\G(\thetah)^{1-\al}] \psi}.
$$
Since $\nabla \G(\thetah)/\G(\thetah)^{1-\al}\in L^2(\Omega)$ by \eqref{MB6}, we need to show $F:=\G(\thet)^{1-\al}-\G(\thetah)^{1-\al}$ belongs to $ L^2(0,T;W^{1,2}_0(\Omega))$ to proceed as in the previous cases. Obviously $F=0$ on $(0,T)\times\partial \Omega$. Moreover, $\nabla[\G(\thet)^{1-\al}-\G(\thetah)^{1-\al}]=(1-\al)[\kappa(\thet)\nabla\thet\G(\thet)^{-\al}-\kappa(\thetah)\nabla\thetah\G(\thetah)^{-\al}]$, $\nabla\thet\G(\thet)^{-\al}\in L^2((0,T)\times\Omega)$ by \eqref{eq:reg-grad-thet} and $\nabla\thetah\G(\thetah)^{-\al}\in L^2((0,T)\times\Omega)$ by \eqref{MB6}. Consequently, we obtain the desired regularity and the fourth term also disappears in the limit as $n\to+\infty$.

Now, taking the limit as $n\to+\infty$ of \eqref{dec-temp} we get 
\begin{equation}\label{dec-temp-after}
\begin{array}{l}\displaystyle\vspace{6pt}
 - \intnO{[\frac{|\vv|^2}{2}+\thet -\thetah - (\mathcal{H}^\al(\thet) - \mathcal{H}^\al(\thetah))\G(\thetah)^{\al}]\, \dert\psi}
\\ \displaystyle\vspace{6pt}
+ \frac{2\al}{(1-\al)^2}\intnO{ \G(\thetah) \left|\nabla\left[\left(\frac{\G(\thet)}{\G(\thetah)}\right)^{\frac{1-\al}{2}}\right]\right|^2 \,\psi} 
\\ \displaystyle\vspace{6pt}
+  \intnO{  \S:\D{\vv}\,  \left(\frac{\G(\thetah)}{\G(\thet)}\right)^{\al} \,\psi}\leq C\intnO{\thet^{1-\al}|\vv|^2\psi}.
\end{array}
\end{equation}
Note that in the first term we added some terms that are independent of $t$ which is allowed since $\psi\in \mathcal{C}^{\infty}_0(0,T)$. The inequality \eqref{dec-temp-after} holds for all $\psi\in \mathcal{C}^\infty_0(0,+\infty)$. However, we want to emphasize that actually, we can allow even $\psi\in \mathcal{C}^{0,1}([0,+\infty))$ with $\psi(0)=0$ and compact support in $[0,+\infty)$, since every such $\psi$ can be approximated by a sequence $\{\psi_n\}\subset \mathcal{C}^\infty_0(0,+\infty)$ such that $\psi_n\to\psi$ and $\psi'_n\to\psi'$ a.e. in $(0,+\infty)$ as $n\to+\infty$, and $\nm{\psi_n}_{1,\infty}\leq2\nm{\psi}_{1,\infty}$.

Now, we recall the auxiliary function $f_\alpha:(0,+\infty)^2\to\R$ from \eqref{DFG} defined by
\begin{equation}\label{definice}
f_\alpha(s,t):=s-t - (\mathcal{H}^\al(s) - \mathcal{H}^\al(t))(\G(t))^{\al}
\end{equation}
for $\alpha\in(0,1)$ and $s,t\in(0,+\infty)$. 
It follows from \cite[Lemma~3.1, Lemma 3.2,(3.31)]{AbBuKa2024} and \eqref{est:lnl1} that for $\alpha\in(0,2/3]$
\begin{equation}\label{Holder}
\begin{split}
\intO{f_\alpha(\thet,\thetah)}\le C\intO{\left| \nabla \frac{\G(\thet)}{\G(\thetah)} \right|^2  \G(\thetah) \left(\frac{\G(\thet)}{\G(\thetah)}\right)^{-1-\al}},
\end{split}
\end{equation}
where the constant $C>0$ depends on $\tmin, \tmax, \nm{\vv_0}_2,\underline{\kappa}, \overline{\kappa},\alpha, \Omega$. The precise dependence of the constant $C$ may be derived from \cite[(3.32)]{AbBuKa2024}.
The inequality \eqref{dec-temp-after} turns into
\begin{equation}\label{dec-temp-after1}
\begin{array}{l}\displaystyle\vspace{6pt}
 - \intnO{\left[\frac{|\vv|^2}{2}+f_\alpha(\thet,\thetah)\right]\partial_t\psi}
\\ \displaystyle\vspace{6pt}
+ c\intnO{ f_{\alpha}(\thet,\thetah)\psi+\G(\thetah) \left|\nabla\left(\frac{\G(\thet)}{\G(\thetah)}\right)^{\frac{1-\al}{2}}\right|^2 \,\psi} 
\\ \displaystyle\vspace{6pt}
 +  \intnO{  \S:\D{\vv}\,  \left(\frac{\G(\thetah)}{\G(\thet)}\right)^{\al} \,\psi}\leq C\intnO{\thet^{1-\al}|\vv|^2 \,\psi}.
\end{array}
\end{equation}
It remains to estimate the last term on the right hand side in order to get the exponential decay of solutions. We split the estimate to the two cases: $p\geq2$ and $p<2$.

Let us start with $p\geq 2$. In this case we closely follow \cite[(3.26)]{AbBuKa2024}. We rewrite the estimate here for readers convenience. We use \eqref{bound-G}, \eqref{MB6}, H\"older's inequality, Sobolev Embedding Theorem, \eqref{MBBounds} and \eqref{eq:kinen1} to get 
\begin{multline}\label{eq:conv-term-fin}
\intnO{\thet^{1-\al}|\vv|^2\psi} \leq C\int_0^{+\infty}{\psi\nm{\vv}_6^2\nm{\thet}_{3(1-\alpha)/2}^{1-\al}}\\\leq C\int_0^{+\infty}{\psi\intO{\S:\D{\vv}}} \leq \frac {C+1}2\int_0^{+\infty}{\dert\psi\nm{\vv}_2^2}-\intT{\psi\intO{\S:\D{\vv}}}.
\end{multline}
We used that by Sobolev Embedding Theorem it holds for the case $p\geq 2$
\begin{equation}\label{eq:sobolev-poincare}
\nm{\vv}_2^2\leq C\nm{\vv}_6^2\leq C\intO{\S:\D{\vv}}.
\end{equation}
We also employed \eqref{est:lnl1} together with the fact that for $\alpha \geq1/3$ we have $3(1-\alpha)/2\leq1$.
Now, using again \eqref{eq:sobolev-poincare} we can find sufficiently large $\beta>0$ and small $\mu>0$ such that application of \eqref{eq:conv-term-fin} in \eqref{dec-temp-after1} gives
$$
-\intnO{\Lab(\vv,\thet,\thetah)\dert\psi}+\mu\intnO{\Lab(\vv,\thet,\thetah)\psi}\leq 0.
$$
Note that $\mu$ and $\beta$ depend on a solution $(\vv,\thet)$ in the manner prescibed in Theorem~\ref{TMB}. In particular, they depend on the solution $(\vv,\thet)$ only through norms of the initial data $(\vv_0,\thet_0)$.
Further we proceed as in Subsection~\ref{sss:324}. We take $n\in\N$, $\tau>0$ a Lebesgue point of the mappings $|\vv|^2,\thet\in L^\infty(0,+\infty;L^1(\Omega))$ and test the equation with $\psi=\exp(\mu t)\max(0,\min(1,n\tau/2-n|t-\tau/2|))$. For the limit passage with $n\to+\infty$ we use that $\partial_2\Lab(r,s,t)\leq 1+(\G(\tmax)/\G(\tmin))^\alpha$ for $s\geq\tmin$ and \eqref{eq:att-ic}. We get
\begin{equation}\label{hle}
\intO{\Lab(\vv(\tau),\thet(\tau),\thetah)\exp(\mu \tau)}\leq \intO{\Lab(\vv_0,\thet_0,\thetah)},
\end{equation}
which is \eqref{assymptotictwo} in the situation $\sigma=0$. General situation of \eqref{assymptotictwo} for a.e. $\sigma,\tau\in (0,+\infty)$, $\sigma<\tau$ is obtained similarly by considering $\sigma,\tau>0$ Lebesgue points of $|\vv|^2,\thet\in L^\infty(0,+\infty;L^1(\Omega))$ and 
$$
\psi(t)=\exp(\mu t)\max(0,\min(1,n[\tau-\sigma-|2t-\tau-\sigma|]/2)).
$$
It remains to show that the estimate \eqref{assymptotictwo} actually holds for a.e. $\sigma>0$ of $\sigma=0$ and all $\tau>\sigma$. This can be done using the scheme of the proof of Theorem~\ref{thm:ex}  and  we sketch the proof here. Since the function $\mathcal{T}_{k,\varepsilon}(\thet)$ is bounded and its time derivative satisfies \eqref{rentime}, we see that it also belongs to $\mathcal{D}_{\textrm{weak}}(0,T; L^2(\Omega))$ and, consequently, that
$$
\mathcal{T}_{k,\varepsilon}(\thet(t)) \rightharpoonup \mathcal{T}_{k,\varepsilon}(\thet(\tau)) \quad \text{weakly in } L^1(\Omega)
$$
for all $\tau \ge 0$ as $t \to \tau_+$. Then, using the convexity of $\Lab$, we deduce that for all $\tau > 0$ (here the importance is that it holds for all $\tau$),
\begin{equation}\label{compl}
\intO{\Lab(\vv(\tau), \mathcal{T}_{k,\varepsilon}(\thet(\tau)), \thetah)} \leq \liminf_{t \to \tau_+} \intO{\Lab(\vv(t), \mathcal{T}_{k,\varepsilon}(\thet(t)), \thetah)}.
\end{equation}
Furthermore, for $k > \overline{\thet} + 1$, we have
$$
\Lab(\vv, \mathcal{T}_{k,\varepsilon}(\thet), \thetah) \le \Lab(\vv, \thet, \thetah).
$$
Since we have already proved above that \eqref{assymptotictwo} holds for almost all $\tau, \sigma$ with $\tau > \sigma$, we also get for almost all $t, \sigma$ with $t > \sigma$ that 
\begin{equation}\label{assymptotictwoQ}
\intO{\Lab(\vv(t), \mathcal{T}_{k,\varepsilon}(\thet(t)), \thetah)} \leq e^{-\mu (t-\sigma)} \intO{\Lab(\vv(\sigma), \thet(\sigma), \thetah)}.
\end{equation}
Hence, using \eqref{compl}, we deduce that for almost all $\sigma$ and for all $\tau > \sigma$,
\begin{equation}\label{assymptotictwoQP}
\intO{\Lab(\vv(\tau), \mathcal{T}_{k,\varepsilon}(\thet(\tau)), \thetah)} \leq e^{-\mu (\tau-\sigma)} \intO{\Lab(\vv(\sigma), \thet(\sigma), \thetah)}.
\end{equation}
Finally, letting $k \to \infty$ in the above inequality, we deduce the validity of \eqref{assymptotictwo}.

In the case $p<2$ we need to proceed more carefully since we cannot use the estimate $\nm{\nabla\vv}_2^2\leq C\intO{\S:\D{\vv}}$. Instead, we have only \eqref{eq:sdv}. It means that if for some $t\in(0,T): \nm{\D\vv(t)}_p\leq1$ we have $\nm{\D\vv(t)}_{p}^2\leq C\intO{\S(t):\D{\vv(t)}}$, and if for some $t\in(0,T): \nm{\D\vv(t)}_p>1$ we get $\nm{\D\vv(t)}_{p}^p\leq C\intO{\S(t):\D{\vv(t)}}$. Now, we provide an estimate for each of the cases mentioned above. In the case $\nm{\D\vv(t)}_p\leq1$ we proceed similarly as if $p\geq 2$. We use the embedding $W^{1,p}(\Omega)\to L^{3p/(3-p)}(\Omega)$, Poincare's and Korn's inequalities and the fact that $3p(1-\alpha)/(5p-6)\leq 1$ if $\alpha\geq2(1/p-1/3)$
\begin{equation}\label{eq:conv-term-fin2}
\intO{\thet^{1-\al}|\vv|^2\psi}\leq C\psi\nm{\vv}_{{3p}/{(3-p)}}^2\nm{\thet}_{3p(1-\alpha)/(5p-6)}^{1-\al}\leq C\psi\intO{\S:\D{\vv}}.
\end{equation}
The case $\nm{\D\vv(t)}_p>1$ we start with an interpolation inequality. We find $q>1$ such that there is $C>0$ that for any $w\in W^{1,p}_0(\Omega)$
$$
\nm{w}_{q}^2\leq C\nm{w}_{2}^{2-p}\nm{w}_{1,p}^p, \quad\mbox{i.e. $q=\frac{12}{12-5p}$}.
$$
Now, we apply this inequality together with the fact that $6(1-\alpha)/(5p-6)\leq 1$ provided $\alpha\geq 2-5p/6$, to get for such $\alpha$'s
\begin{multline*}
\intO{\thet^{1-\al}|\vv|^2\psi}\leq C\psi\nm{\vv}_{12/{(12-5p)}}^2\nm{\thet^{1-\alpha}}_{6/(5p-6)}\\
\leq C\psi\nm{\vv}_{2}^{2-p}\nm{\vv}_{1,p}^p\nm{\thet^{1-\alpha}}_{6/(5p-6)}\leq C\psi\intO{\S:\D{\vv}}.
\end{multline*}
We use these estimates on different time levels $t\in(0,T)$ according to the size of the norm $\nm{\D\vv(t)}_p$. This is the reason why we need $\alpha\geq\max(2-5p/6,2(1/p-1/3))=2-5p/6$ for $p\geq6/5$. The other restrictions on $\alpha$ tell us that $\alpha\in(1/2,2/3]$, compare \eqref{Holder} and the estimate of $|A_{22}|$ below \eqref{B}. The set of eligible $\alpha$'s is nonempty if $p\geq 8/5$. Altogether, we get in this situation again
\begin{multline*}
  C\intnO{\thet^{1-\al}|\vv|^2\,\psi}\leq C\intnn{\psi\intO{\S:\D{\vv}}} \\
  \leq \frac {C+1}2\int_0^{+\infty}{\dert\psi\nm{\vv}_2^2}-\intT{\psi\intO{\S:\D{\vv}}}
\end{multline*}
which gives after plugging into \eqref{dec-temp-after1}
\begin{equation}\label{eq:final-ljap}
\begin{array}{l}\displaystyle\vspace{6pt}
 - \intnO{[\frac{|\vv|^2}{2}+f_\alpha(\thet,\thetah)]\partial_t\psi}
+ c\intnO{ (\S:\D{\vv}+f_{\alpha}(\thet,\thetah))\psi}\leq 0,
\end{array}
\end{equation}
where the constant $c>0$ depends on the solution only through $R>0$ in \eqref{assymptotictwo}.  
The final difficulty we encounter is again related to the estimate \eqref{eq:sdv}. To address this, we introduce
\[
  0 \leq \tstar := \inf\{ \tau \; ; \; \nm{\vv(\tau)}_{2} < 1 \},
\]
and estimate
$$
\begin{aligned}
\nm{\vv(t)}_2^2&\leq C\intO{\S:\D{\vv}}\quad&&\mbox{for $t\in(\tstar,+\infty)$,}\\
\nm{\vv(t)}_2^2&\leq {\nm{\vv(t)}_2^p}{\nm{\vv_0}}_2^{2-p} \leq C\nm{\vv_0}_2^{2-p}\intO{\S:\D{\vv}}\quad&&\mbox{for $t\in(0,\tstar)$.}
\end{aligned}
$$
Using this estimate in \eqref{eq:final-ljap} we conclude the proof as in the situation $p\geq2$. Theorem~\ref{TMB} is proved. \qed

\appendix
\section{Auxiliary tools}

\begin{lemma}[Decay estimates]\label{lem:decay} \ 
\begin{itemize}[left=0pt]
\item[1)]
Let $\tstar \geq 0$ and let $f : (\tstar, +\infty) \to [0, +\infty]$ be a measurable function such that, for almost all $s, t \in (\tstar, +\infty)$ with $s < t$, there holds
\begin{equation}\label{eq:0}
f(t) + C_1 \int_s^t f(\tau) \, \dtau \leq f(s).
\end{equation}
Then
$$
f(t) \leq f(s) \, e^{-C_1 (t-s)}
$$
for almost all $s, t > \tstar$ with $s < t$.

\item[2)] Let $\tstar > 0$, $\alpha \in (0,1)$, and let $f : (0, \tstar) \to [0, +\infty]$ be measurable such that, for almost all $s, t \in (0, \tstar)$ with $s < t$, there holds
\begin{equation}\label{eq:1}
\max\Bigl(1, f(t) + C_1 \int_s^t f(\tau)^\alpha \, \dtau \Bigr) \leq f(s).
\end{equation}
Then
$$
f(t) \leq \left(\frac{f(s)}{1 + C_1 (t-s)}\right)^{1/\alpha}
$$
for almost all $s, t \in (0, \tstar)$ with $s < t$.

\end{itemize}
\end{lemma}
\begin{proof}
  Ad 1)
  We redefine $f$ on a set of zero measure so that $f$ is decreasing on $(\tstar,+\infty)$. We denote the new function again $f$. It satisfies  \eqref{eq:0} for all $s,t>\tstar$, $s<t$. Since $f$ is decreasing we get from \eqref{eq:0} that 
  \begin{equation}\label{eq:2}
  f(\tau)\leq\frac{f(\sigma)}{1+C_1(\tau-\sigma)}
\end{equation}
for all $\sigma,\tau\in(\tstar,+\infty)$, $\sigma<\tau$.
Splitting interval $(s,t)$ to $n\in\N$ parts and applying \eqref{eq:2} iteratively gives
$$
  f(t)\leq f(s)\left(\frac1{1+C_1\frac{t-s}n}\right)^n.
$$
The claim follows by limit passage $n\to+\infty$ since we modified the function $f$ only on a set of zero measure.

Ad 2) As in 1) we redefine $f$ on a set of zero measure so that $f$ is decreasing and satisfies \eqref{eq:1} for all $s,t\in(0,\tstar)$. Since $f\geq 1$ on $(0,\tstar)$ we get for all $s,t\in (0,\tstar)$, $s<t$
  $$
  f(t)\leq \left(\frac{f(s)}{1+C_1(t-s)}\right)^{\frac1\alpha},
  $$
  which is the claim since $f$ was modified only on a zero measure set.  
\end{proof}

\begin{theorem}\label{thm:hardy}
  Let $1<p<+\infty$ and suppose that $\Omega\subset\R^d$ is a Lipschitz, open, bounded set. If $u\in W^{1,p}_0(\Omega)$, then there is a constant $C>0$, depending on $p$, $d$ and $\Omega$, such that
  $$
  \int_{\Omega}\left(\frac{|u(x)|}{\dist(x,\partial\Omega)}\right)^p\dx\leq
  C\int_{\Omega}|\nabla u(x)|^p\dx.
  $$
\end{theorem}
Theorem is proved in \cite[Corollary~3.11]{KiMa1997} under a more general conditions on the set $\Omega$. Origin of the Hardy inequality goes back to \cite{HaLiPo1952} and \cite{Ne1962}.

\section{Disclosure statement}
The authors report there are no competing interests to declare.

\bibliographystyle{amsplain}
\bibliography{bibliography}
\end{document}